\documentclass[10pt]{amsart}

\newtheorem{theorem}[subsubsection]{Theorem}
\newtheorem{lemma}[subsubsection]{Lemma}
\newtheorem{proposition}[subsubsection]{Proposition}

\newtheorem{conjecture}[subsubsection]{Conjecture}

\theoremstyle{definition}

\theoremstyle{remark}
\newtheorem{remark}[subsubsection]{Remark}

\DeclareMathAlphabet{\mathpzc}{OT1}{pzc}{m}{it}
\usepackage{amsmath}
\usepackage{amssymb}
\usepackage{euscript}
\usepackage{array}
\usepackage[all]{xy}
\usepackage{amscd}

 \newcommand{\FF}{{\mathbb F}}
 
 \newcommand{\ZZ}{{\mathbb Z}}
 \newcommand{\QQ}{{\mathbb Q}}
 \newcommand{\CC}{{\mathbb C}}
 \newcommand{\NN}{{\mathbb N}}

 \newcommand{\GG}{{\mathbb G}}

\newcommand{\EE}{{\mathbb E}}
\newcommand{\TT}{{\mathbb T}}
\newcommand{\LL}{{\mathbb L}}
\renewcommand{\a}{\alpha}   
\renewcommand{\d}{\delta}  \renewcommand{\th}{\theta}
  \renewcommand{\k}{\kappa}
\renewcommand{\l}{\lambda}  \renewcommand{\th}{\theta}
    
 \newcommand{\s}{\sigma} \renewcommand{\o}{\omega}
  \renewcommand{\t}{\tau}
\newcommand{\G}{\Gamma}  \renewcommand{\L}{\Lambda}

\DeclareMathOperator{\trdeg}{tr.deg} 
\DeclareMathOperator{\Ker}{Ker} 
\DeclareMathOperator{\GL}{GL} \DeclareMathOperator{\Mat}{Mat}
 \DeclareMathOperator{\End}{End}

\DeclareMathOperator{\Res}{Res} \DeclareMathOperator{\Ext}{Ext}

\begin{document}
\title[Periods of third kind and algebraic independence]{Periods of third kind for rank 2 Drinfeld modules  and algebraic independence
of logarithms}

\author[Chieh-Yu Chang]{Chieh-Yu Chang}
\address{Mathematics Division, National Center for Theoretical Sciences,
National Tsing Hua University, Hsinchu City 30042, Taiwan
  R.O.C.}
\address{Department of Mathematics, National Central
  University, Chung-Li 32054, Taiwan R.O.C.}

 \email{cychang@math.cts.nthu.edu.tw}
%\thanks{}

\date{June 25, 2009}

\thanks{The author was supported by NCTS}

\begin{abstract}
In analogy with the periods of abelian integrals of differentials of
third kind for an elliptic curve defined over a number field, we
introduce a notion of periods of third kind for a rank $2$ Drinfeld
$\FF_{q}[t]$-module $\rho$ defined over an algebraic function field
and derive explicit formulae for them. When $\rho$ has complex multiplication by a separable extension, we prove the
algebraic independence of $\rho$-logarithms of algebraic points that
are linearly independent over the CM field of $\rho.$ Together with
the main result in \cite{CP08}, we completely determine all the
algebraic relations among the periods of first, second and third
kinds for rank $2$ Drinfeld $\FF_{q}[t]$-modules in odd
characteristic.

\end{abstract}

\maketitle
\section{Introduction}
\subsection{Motivation}
Let $E$ be an elliptic curve defined over $\overline{\QQ}$ given by
$y^{2}=4x^{3}-g_{2}x-g_{3}$. Let $\L=\ZZ \o_{1}+\ZZ \o_{2}$ be the
period lattice of $E$ comprising the periods of the holomorphic
differential form $ dx/ y $ (of the first kind) and let
$\wp(z),\zeta(z),\s(z)$ be the Weierstrass functions associated to
$\L$. Let $\eta:\L\rightarrow \CC$ be the quasi-period function
defined by $\eta(\o):=\zeta(z+\o)-\zeta(z)$. The quasi-periods
$\eta(\o)$, for $\o\in \L$, arise from the differential form $x dx /
y$ of the second kind (having poles with zero residues) and occur in
the second coordinate of period vectors of the exponential function
$(z_1,z_2)\mapsto (1,\wp (z_1),\wp'(z_1),z_2-\zeta (z_1) )$ for a
commutative algebraic group as an extension of $E$ by the additive
group $\GG_{a}$. In the 1930s, Schneider established a fundamental
theorem which asserts that nonzero periods and quasi-periods of $E$
are transcendental over $\QQ$. The periods and quasi-periods satisfy
the Legendre
relation: $$ {\rm{det}}\left(%
\begin{array}{cc}
  \o_{1} & \o_{2} \\
  \eta(\o_1) & \eta(\o_2) \\
\end{array}%
\right)=\pm 2 \pi \sqrt{-1}. $$ If the elliptic curve $E$ has
complex multiplication, then the transcendence degree of the field
$$ \overline{\QQ}(\o_1,\o_2,\eta(\o_1),\eta(\o_2)  )  $$
over $\overline{\QQ}$ is known to be $2$ by a theorem of Chudnovsky.
When $E$ has no complex multiplication, conjecturally one expects
that $\o_1,\o_2,\eta(\o_1),\eta(\o_2)$ are algebraically independent
over $\QQ$.

Given $u\in \CC$ so that $\wp(u)\in \overline{\QQ}$, we regard
$(1,\wp(u),\wp'(u))\in E^{\vee}\cong \Ext^{1}(E,\GG_{m} )$ and
 consider the differential form of the third kind (having poles with nonzero residues):
$$ \d=\frac{1}{2} \frac{y+ \wp'(u) }{x- \wp(u)  } \frac{dx}{y} .$$
For $\o\in \L$, we set $\l(\o,u):=\o \zeta(u)-\eta(\o)u$. If
$\gamma$ is any closed cycle on $E(\CC)$ along which $\d$ is
holomorphic, then the period of the integral of $\d$ along the cycle
takes the form
\begin{equation}\label{form of PThird intro}
 \l(\o,u)+ 2m \pi \sqrt{-1}
\end{equation}
for some $\o \in \L$ and $m\in \ZZ$. The values $\l(\o,u)+2m \pi
\sqrt{-1}$ as above occur in the second coordinate of  period
vectors  of the exponential function $$\left(z_1,z_2 \right)\mapsto
\left(1, \wp(z_1),\wp'(z_1),\frac{\s(z_1-u)  }{\s(z_1)\s(u) }
{\rm{exp}}(\zeta (u)z_1-z_2) \right) $$ for a commutative algebraic
group as an extension of $E$ by the multiplicative group $\GG_{m}$.
The transcendence of $\l(\o,u)+ 2m \pi \sqrt{-1}$ was established by
Laurent (cf. \cite{L80, L82}) if it is nonzero, and later on
W\"ustholz \cite{W84} extended Laurent's result to arbitrary
differentials of the third kind  over $\overline{\QQ}$. Given
nonzero $\o\in \L$ and $u_1,\ldots,u_n \in \CC$ with $\wp(u_i)\in
\overline{\QQ} $ for $1 \leq i\leq n$, the $\overline{\QQ}$-linear
independence of the following periods (of three different kinds)
$\left\{ \o_,\eta(\o),\l(\o,u_1),\ldots,\l(\o,u_n) \right\}$ was
achieved by W\"ustholz \cite{W84} if $\o,u_{1},\ldots,u_{n}$ are
linearly independent over $\QQ$ (see also \cite{BW07}). To determine
all the $\overline{\QQ}$-algebraic relations among these periods in
question, it suffices to prove the following conjecture:

\begin{conjecture}\label{conjecture}
Let $u_1,\ldots,u_n\in \CC$ satisfy $\wp(u_i)\in \overline{\QQ}$ for
$i=1,\ldots,n$. If $u_1,\ldots,u_{n}$ are linearly independent over
the endomorphism ring of $E$, then the $2n$ numbers
$$ u_1,\ldots,u_{n},\zeta(u_1),\ldots,\zeta(u_n)  $$
are algebraically independent over $\overline{\QQ}$.
\end{conjecture}

But even the weaker conjecture that the values $u_{1},\ldots,u_{n}$
are algebraically independent over $\overline{\QQ}$ is still open.

\subsection{The main results}
The purpose of the paper are two themes. The first is to introduce
periods of third kind for rank $2$ Drinfeld modules and develop
explicit formulae for them. The second is to prove the analogue of
Conjecture \ref{conjecture} in the setting of rank $2$ Drinfeld
modules with complex multiplication by separable extensions. (Note that a rank $2$ Drinfeld module
with complex multiplication by an inseparable extension exists only in characteristic $2$).

Let $\FF_q$ be the finite field of $q$ elements, where $q$ is a
power of a prime $p$. Let $k:=\FF_q(\theta)$ be the rational
function field in a variable $\theta$ over $\FF_q$. Let
$\CC_{\infty}$ be the completion of an algebraic closure of
$\FF_q((\frac{1}{\theta})) $ with respect to the non-archimedean
absolute value of $k$ for which $|\theta|_{\infty}=q$. Let
$\t:\CC_{\infty}\rightarrow \CC_{\infty}$ be the Frobenius operator
($x\mapsto x^{q}$) and let $\CC_{\infty}[\tau]$ be the twisted
polynomial ring in $\tau$ over $\CC_{\infty}$ subject to the
relation $\tau c=c^{q} \tau$ for $c\in \CC_{\infty}$.

A Drinfeld $\FF_{q}[t]$-module (of generic characteristic) is an
$\FF_q$-linear ring homomorphism $\rho: \FF_{q}[t]\rightarrow
\CC_{\infty}[\t]$ so that the coefficient of $\t^{0}$ in $\rho_{t}$
is $\th$ and  $\rho_{t}\notin \CC_{\infty}$. The degree of $\rho_t$
in $\t$ is called the rank of $\rho$. Moreover, the exponential
function of $\rho$ is defined to be the unique power series of the
form $\exp_{\rho}(z)=z+\sum_{i=1}^{\infty}\a_i z^{q^{i}}$, with $\a_i \in
\CC_{\infty}$, satisfying the functional equation $\exp_{\rho}(\th
z)=\rho_{t}(\exp_{\rho}(z)  ) $. One can show that $\exp_{\rho}$ is
an entire function on $\CC_{\infty}$ and its kernel $\L_{\rho}:=\Ker
\exp_{\rho}$ is a discrete, $\FF_{q}[\th]$-module of rank $r$ inside
$\CC_{\infty}$. This $\L_{\rho}$ is called the period lattice of
$\rho$ and elements of $\L_{\rho}$ are called periods (of first
kind) of $\rho$. Finally, we say that $\rho$ is defined over
$\bar{k}$ if the coefficients of $\rho_t$ lie in $\bar{k}$.

The ring $\End(\rho):= \left\{\a\in \CC_{\infty};\hbox{ }\a
\L_{\rho}\subseteq \L_{\rho} \right\}$ is called the multiplication
ring of $\L_{\rho}$ and it can be identified with the endomorphism
ring of $\rho$ (cf. \cite{G96, R02, T04}).  Given two Drinfeld
$\FF_{q}[t]$-modules $\rho$ and $\nu$, we say that $\rho$ is
isomorphic to $\nu$ if there exists $\epsilon\in
\CC_{\infty}^{\times}$ for which $\nu_{t}=\epsilon^{-1}{\rho}_{t}
\epsilon$.

Fix a rank $2$ Drinfeld $\FF_{q}[t]$-module $\rho$  given by
$\rho_{t}=\th+\kappa \t+ \Delta \t^{2}$ with $\kappa, \Delta\in
\bar{k}$, $\Delta\neq 0$. Let
$\L_{\rho}=\FF_{q}[\th]\o_{1}+\FF_{q}[\th]\o_{2}$ be the period
lattice of $\rho$. When
$\End(\rho)\supsetneq \FF_{q}[\th]$, we say that $\rho$ has complex
multiplication and the fraction field of $\End(\rho)$ is called the
CM field of $\rho$. If $\End(\rho)=\FF_{q}[\th]$, then we say that
$\rho$ has no complex multiplication.

 Let $F_{\tau}$ be the quasi-periodic function
of $\rho$ associated to $\tau$, i.e., it is the unique power series
satisfying the conditions:
\begin{enumerate}
\item[$\bullet$] $F_{\t}(z)\equiv 0 \hbox{ }({\rm{mod}}\hbox{ }z^{q}   );$
\item[$\bullet$] $F_{\t}(\th z)-\th F_{\t}(z)=\exp_{\rho}(z)^{q}$.
\end{enumerate}
The function $F_{\t}$ has the following properties:
\begin{enumerate}
\item[$\bullet$] $F_{\t}$ is entire on $\CC_{\infty}$;
\item[$\bullet$] $F_{\t}(z+\o)=F_{\t}(z)+F_{\t}(\o)$ for $\o\in \L_{\rho}$;
\item[$\bullet$] ${F_{\t}}\mid_{\L_{\rho}}: \L_{\rho}\rightarrow \CC_{\infty}$ is $\FF_{q}[\th]$-linear.
\end{enumerate}
The values $F_{\tau}(\o)$ for $\o\in \L_{\rho}$, are called
quasi-periods of $\rho$ associated to $\tau$ (or periods of second
kind for $\rho$).

For each $\varphi=\sum_{i}a_{i}\t^{i} \in \Mat_d(\CC_{\infty}[\t])$
with $a_{i}\in \Mat_{d}(\CC_{\infty})$, we put $\partial
\varphi:=a_{0}$. A $t$-module of dimension $d$ is an
$\FF_{q}$-linear ring homomorphism
$$\phi:\FF_{q}[t]\rightarrow \Mat_{d}(\CC_{\infty}[\t])
$$ so that $\partial \phi_{t}-\th I_{d}$ is a nilpotent matrix. Its exponential function, denoted by $\exp_{\phi}$, is
defined to be the unique entire function as $\FF_q$-linear
homomorphism from $\CC_{\infty}^{d}$ to
 $\CC_{\infty}^{d}$, satisfying the conditions:
\begin{enumerate}
\item[$\bullet$] $\exp_{\phi}(\mathbf{z})\equiv \mathbf{z}$ (mod deg $q$);
\item[$\bullet$] $\exp_{\phi}(\partial \phi_{t}
(\mathbf{z}))=\phi_{t}\left(\exp_{\phi}(\mathbf{z}) \right)$.
\end{enumerate}
We say that $\phi$ is uniformizable if $\exp_{\phi}$ is surjective
onto $\CC_{\infty}^{d}$ (cf. \cite{A86}).

Now, we turn to the quasi-periods of $\rho$. In fact, for $\o\in
\L_{\rho}$ the vector $\left(\o, -F_{\t}(\o)\right)^{tr}$ is a
period vector of $\exp_{\phi}$ for the two dimensional $t$-module
$\phi$ defined by $$\phi_{t}=\left(
                                                                                \begin{array}{cc}
                                                                                  \rho_{t} & 0 \\
                                                                                  \t & \th \\
                                                                                \end{array}
                                                                              \right),
$$ which is an extension of $\rho$ by the additive group $\GG_{a}$
(cf. \cite{Ge89, Yu90, BP02}).

In \cite{Yu86} and \cite{Yu90}, Yu established fundamental results
parallel to the work of Schneider. That is, nonzero periods and
quasi-periods of $\rho$ are transcendental over $k$. Anderson
proved an analogue of the Legendre relation:
$$ \o_{1}F_{\t}(\o_{2})-\o_{2}F_{\t}(\o_{1})  =\tilde{\pi}/ \sqrt[q-1]{-\Delta},$$ where $\tilde{\pi}$ is a fundamental
period of the Carlitz module $C$, which is the rank one Drinfeld $\FF_{q}[t]$-module defined by $C_{t}=\th+\t$, and $\sqrt[q-1]{-\Delta}$ is a choice of $(q-1)$st root of $-\Delta$, which is fixed throughout this paper. If
$\rho$ has complex multiplication, then the transcendence degree of
the field $$\bar{k}(\o_1,\o_2,F_{\t}(\o_1),F_{\t}(\o_2))$$ over
$\bar{k}$ is known to be $2$ by  Thiery \cite{Thi92}. When $p$ is
odd and $\rho$ has no complex multiplication, the author and
Papanikolas \cite{CP08} have proved the algebraic independence of
$\o_1,\o_2,F_{\t}(\o_1),F_{\t}(\o_2)$ over $\bar{k}$.

Consider the category of two-dimensional uniformizable $t$-modules
which are extensions of $\rho$ by the Carlitz module $C$ and whose
Lie algebras are split (see \cite{PR03} and $\S 2$). We denote the
category by $\Ext^{1}_{0}(\rho,C)$ and note that it is in bijection
with $\left\{ \a \tau; \hbox{ }\a\in \CC_{\infty} \right\}$.
Moreover, $\Ext^{1}_{0}(\rho,C)$ has a $t$-module structure
isomorphic to the rank $2$ Drinfeld $\FF_{q}[t]$-module $\rho$.
Fixing any $\a\in \bar{k}^{\times}$, we consider $\a \tau \in
\Ext^{1}_{0}(\rho,C)$ that corresponds to the uniformizable
$t$-module $\phi$ defined by
\begin{equation}\label{phi intro}
\phi_t= \left(%
\begin{array}{cc}
 \rho_{t}  & 0 \\
  \a \t & C_{t} \\
\end{array}%
\right).
\end{equation}
Note that any period vector in $\hbox{Ker exp}_{\phi}$
can be written in the form $\left(%
\begin{array}{c}
  \o \\
  \l \\
\end{array}%
\right)$  for some $\o\in \L_{\rho}$ and $\l \in \CC_{\infty}$ (cf.
$\S$\ref{subsec for 3rd kind}). We call such $\l$ a period of third
kind for $\rho$ (associated to $\a \tau$). The first main theorem of the present paper is to
derive an explicit formula for $\l$ as an analogue of (\ref{form of
PThird intro}):

\begin{theorem}\label{thm 1 intro}
Given $\a\in \bar{k}^{\times}$, let $\phi$ be the $t$-module defined
in {\rm{(\ref{phi intro})}}.  Let $u\in \CC_{\infty}$ satisfy
${\rm{exp}}_{\rho}(u)= \a /\sqrt[q-1]{-\Delta}   $. For any period
vector $\left(%
\begin{array}{c}
  \o \\
  \l \\
\end{array}%
\right)\in {\rm{Ker}}\hbox{ }{\rm{exp}}_{\phi}$, there exists $f\in
\FF_{q}[t]$ so that
$$ \l=- \sqrt[q-1]{-\Delta}  \left( u F_{\t}(\o) - \o F_{\t}(u)  \right)+f(\th) \tilde{\pi}  .$$

\end{theorem}

According to the formula in Theorem \ref{thm 1 intro}, to determine
all the algebraic relations among the periods of three different
kinds for $\rho$ is equivalent to prove the analogue of Conjecture
\ref{conjecture}. When $p$ is odd and $\rho$ has no complex
multiplication, the author and Papanikolas \cite[Thm. 1.2.4]{CP08}
have proved the algebraic independence of
$u_1,\ldots,u_n,F_{\t}(u_1),\ldots,F_{\t}(u_n)$ if $u_1,\ldots,u_n$
are linearly independent over $k$, where $u_1,\ldots,u_n \in
\CC_{\infty}$ satisfy $\hbox{exp}_{\rho}(u_i)\in \bar{k}$ for
$i=1,\ldots,n$. The second main theorem in this paper is to
establish the counterpart:
\begin{theorem}\label{thm 2 intro}
Let $\rho$ be a rank $2$ Drinfeld
$\FF_{q}[t]$-module with complex multiplication defined over
$\bar{k}$ and suppose that its CM field is separable over $k$. Let $u_{1},\ldots,u_{n}\in \CC_{\infty}$ satisfy
${\rm{exp}}_{\rho}(u_{i})\in \bar{k}$ for $i=1,\ldots,n$. If
$u_{1},\ldots,u_{n}$ are linearly independent over
${\rm{End}}(\rho)$, then the $2n$ quantities
$$u_{1},\ldots,u_{n},F_{\t}(u_{1}),\ldots,F_{\t}(u_{n}) $$
are algebraically independent over $\bar{k}$.
\end{theorem}

The omitted case of the rank $2$ Drinfeld modules with complex multiplication by inseparable extensions
imposes some added difficulties, which require further investigation (see Remark \ref{rmk final subsec}). Combining Theorem \ref{thm 2 intro} and \cite[Thm. 1.2.4]{CP08} we
obtain the algebraic independence of periods of three different
kinds in odd characteristic:
\begin{theorem}\label{thm 3 intro}
Let $p$ be odd and let $\rho$ be a rank $2$ Drinfeld
$\FF_{q}[t]$-module defined over $\bar{k}$. Let
$u_{1},\ldots,u_{n}\in \CC_{\infty}$ satisfy
${\rm{exp}}_{\rho}(u_{i})\in \bar{k}$ for $i=1,\ldots,n$. Given a
nonzero period $\o\in \L_{\rho}$, we set $\l(\o,u_{i}):=\o
F_{\t}(u_{i})-u_{i} F_{\t}(\o)$ for $i=1,\ldots,n$. If $ \o,
u_{1},\ldots,u_{n}$ are linearly independent over
${\rm{End}}(\rho)$, then the $2+n$ quantities
$$  \o, F_{\t}(\o),\l(\o,u_{1}),\ldots,\l(\o,u_{n}) $$
are algebraically independent over $\bar{k}$.

\end{theorem}

\subsection{Methods and outline of this paper}
In order to derive the explicit formula in Theorem \ref{thm 1
intro}, we consider the Anderson generating function of a period
vector for the $t$-module $\phi$ given in (\ref{phi intro}). The
function satisfies a difference equation. The subtle part here is to
prove that such a difference equation is related to a difference
equation associated to the $\rho$-logarithm of an algebraic point
(i.e., its image under $\exp_{\rho}$ lies in $\bar{k}$), whence we
obtain Theorem \ref{thm 1 intro}.

To prove Theorem \ref{thm 2 intro}, it suffices to prove the
algebraic independence of
$$\o_1,u_1,\ldots,u_n,F_{\t}(\o_1),F_{\t}(u_1),\ldots,F_{\t}(u_n) $$
in the situation that $\Delta=1$ and $\o_1,u_1,\ldots,u_n$ are
linearly independent over $\hbox{End}(\rho)$. To prove it, we shall
apply the fundamental theorem of Papanikolas \cite{P08}, which is a
function field analogue of Grothendieck's periods conjecture for
abelian varieties defined over $\overline{\QQ}$. More precisely, for each $1\leq i\leq n$, we construct a suitable $t$-motive
$M_{i}$, which is an extension of the trivial $t$-motive by the $t$-motive $M_{\rho}$ associated to $\rho$, so that
its period matrix contains the periods, quasi-periods of $\rho$, $u_{i}$ and $F_{\t}(u_{i})$. Here a $t$-motive is a dual notion of a $t$-module introduced by Anderson \cite{A86}. We further define the $t$-motive $M$
as the direct sum $\oplus_{i=1}^{n}M_{i}$. Then we are reduced to proving that the dimension of the motivic Galois
group $\G_{M}$ of $M$ is equal to $2+2n$.

The first step of the proof is to show that the motivic Galois group
$\G_{M_{\rho}}$ of the $t$-motive $M_{\rho}$ is
isomorphic to the restriction of scalars of $\GG_{m}$ from the CM
field of $\rho$. It follows that $\G_{M}$  is an extension of the torus $\G_{M_{\rho}}$ by its unipotent radical. Next, under the hypothesis of linear independence of the quantities in question, we prove the $\rm{End}_{\mathcal{T}}(M_{\rho})$-linear independence of the $t$-motives $M_{1},\ldots,M_{n}$  in $\rm{Ext}^{1}_{\mathcal{T}}(\mathbf{1},M_{\rho})$ (see Theorem \ref{K independence in Ext}), where $\mathcal{T}$ is the category of $t$-motives which is a neutral Tannakian category over $\FF_{q}(t)$ developed by Papanikolas. Having this nice property at hand, we follow Hardouin (\cite[Thm. 4.7]{Pe07}, \cite[Cor. 2.4]{H09}) to prove that if the dimension of the unipotent radical of $\G_{M}$ is less than $2n$, then $M_{1},\ldots,M_{n}$ are linearly dependent over $\rm{End}(M_{\rho})$ in $\rm{Ext}^{1}_{\mathcal{T}}(\mathbf{1},M_{\rho})$, which is a contradiction and hence we prove Theorem \ref{thm 2 intro}.

The present paper is organized as follows. In $\S 2$, we review
Papanikolas' theory and prove that $\G_{M_{\rho}}$ is a non-split
torus over $\FF_{q}(t)$ under the  hypothesis on $\rho$. In $\S 3$, we review the
work of Papanikolas-Ramachandran about the relationship between
$(\rho,C)$-biderivations and $\Ext^{1}_{0}(\rho,C)$, and further
prove Theorem \ref{thm 1 intro}. Finally, we prove Theorem \ref{thm 2 intro} and derive Theorem \ref{thm 3 intro} in
$\S4$.

\section{Periods and motivic Galois groups}

\subsection{Notation and definitions} \label{notation}

\subsubsection{Table of symbols.}${}$
\\$\mathbb{F}_{q}:=$ the finite field of $q$ elements, $q$ is a power
of a prime number $p$.
\\$\th,t:=$ independent variables.
\\$\mathbb{F}_{q}[\theta]:=$ the polynomial ring in the variable
$\theta$ over $\mathbb{F}_{q}.$
\\$k:=\mathbb{F}_{q}(\theta)$, the fraction field of $\FF_{q}[\th]$.
\\$k_{\infty}:=\mathbb{F}_{q}((\frac{1}{\theta}))$, the completion of
$k$ with respect to the place at infinity.
\\$\overline{k_{\infty}}:=$ a fixed algebraic closure of
$k_{\infty}$.
\\$\overline{k}:=$ the algebraic closure of $k$ in
$\overline{k_{\infty}}$.
\\$\mathbb{C}_{\infty}:=$ the completion of $\overline{k_{\infty}}$
with respect to canonical extension of the place at infinity.
\\$|\cdot|_{\infty}:=$ a fixed absolute value for the completed field
$\mathbb{C}_{\infty}$ so that $|\th|_{\infty}=q$.
\\$\mathbb{T}:=\{f\in \mathbb{C}_{\infty}[[t]]$; f${}$ converges${}$
on${}$ $|t|_{\infty}\leq_{}1 \}$, the Tate algebra.
\\$\mathbb{L}:=$ the fraction field of $\mathbb{T}.$
%\\${\GG}_{a}:=$ the additive group.
\\$\hbox{GL}_r:=$ group of invertible rank $r$  matrices.
%\\${\GG}_{m}:=\hbox{GL}_{1}$, the multiplicative group.
%\\$A^{\vee}:=(A^{-1})^{tr},$ for $A\in \GL_{r}(\LL)$.
%\\$I_{n}:=$ the identity matrix of size $n$.
\subsubsection{Frobenius twisting and entire power series}\label{sec Frob
Twist}

For $n\in \ZZ$, given a Laurent series $f = \sum_{i}a_{i}t^{i} \in
\CC_{\infty}((t))$ we define the $n$-fold twist of $f$ by the rule
$\s^{-n}f:= f^{(n)}:=\sum_{i}a_{i}^{q^{n}}t^{i}$.  For each $n$, the
twisting operation is an automorphism of the Laurent series field
$\CC_{\infty}((t))$ stabilizing several subrings, e.g.,
$\bar{k}[[t]]$, $\bar{k}[t]$ and $\TT$. More generally, for any
matrix $B$ with entries in $\CC_{\infty}((t))$ we define $B^{(n)}$
by the rule ${B^{(n)}}_{ij}:=B_{ij}^{(n)}$.

A power series $f = \sum_{i=0}^{\infty} a_{i}t^{i}\in
\CC_{\infty}[[t]]$ that satisfies
\[
\lim_{i \rightarrow \infty} \sqrt[i]{|a_{i}|_{\infty} }=0 \quad
\textnormal{and}\quad
[k_{\infty}(a_{0},a_{1},a_{2},\cdots):k_{\infty} ]< \infty
\]
is called an entire power series. As a function of $t$, such a power
series $f$ converges on all of $\CC_{\infty}$ and, when restricted
to $\overline{k_{\infty}}$, $f$ takes values in
$\overline{k_{\infty}}$. The ring of entire power series is denoted
by $\EE$.

\subsection{Review of Papanikolas' theory} In this subsection we follow \cite{P08} for
background and terminology of $t$-motives.  Let
$\bar{k}(t)[\sigma,\sigma^{-1}]$ be the noncommutative ring of
Laurent polynomials in $\sigma$ with coefficients in $\bar{k}(t)$,
subject to the relation
\[
\sigma f=f^{(-1)}\sigma, \quad \forall\,f \in \bar{k}(t).
\]
Let $\bar{k}[t,\s]$ be the noncommutative subring of
$\bar{k}(t)[\s,\s^{-1}]$ generated by $t$ and $\s$ over $\bar{k}$.
An Anderson $t$-motive is a left $\bar{k}[t,\s]$-module
$\mathcal{M}$ which is free and finitely generated both as a left
$\bar{k}[t]$-module and a left $\bar{k}[\s]$-module and which
satisfies, for integers $N$ sufficiently large,
\begin{equation}\label{cond for Anderson t-motives}
(t-\th)^{N}\mathcal{M} \subseteq \s \mathcal{M}.
\end{equation}
Given an Anderson $t$-motive $\mathcal{M}$ of rank $r$ over
$\bar{k}[t]$, let $\mathbf{m}\in \Mat_{r\times 1}(\mathcal{M})$
comprise a $\bar{k}[t]$-basis of $\mathcal{M}$. Then multiplication
by $\s$ on $\mathcal{M}$ is represented by $\s \mathbf{m}=\Phi
\mathbf{m}$ for some $\Phi\in \Mat_{r}(\bar{k}[t])$. Note that
$(\ref{cond for Anderson t-motives})$ implies $\det
\Phi=c(t-\th)^{s}$ for some $c\in \bar{k}^{\times}$ and $s\geq 0$.

We say that a left $\bar{k}(t)[\s,\s^{-1}]$-module $P$ is a
pre-$t$-motive if it is finite dimensional over $\bar{k}(t)$. Let
$\mathcal{P}$ be the category of pre-$t$-motives. Morphisms in
$\mathcal{P}$ are left $\bar{k}(t)[\s,\s^{-1}]$-module homomorphisms
and hence $\mathcal{P}$ forms an abelian category. The identity
object of $\mathcal{P}$ is denoted by $\mathbf{1}$ whose underlying
space is $\bar{k}(t)$ together with the $\s$-action given by
$$ \s f=f^{(-1)}\hbox{ for }f\in \mathbf{1}    .$$

In what follows we are interested in pre-$t$-motives $P$ that are rigid analytically trivial. To define this,
we let $\mathbf{p}\in {\rm{Mat}}_{r\times 1}(P)$ comprise
a $\bar{k}(t)$-basis of $P$. Then multiplication by $\s$ on $P$ is
given by $\s \mathbf{p}=\Phi \mathbf{p}$ for some $\Phi\in
\GL_{r}(\bar{k}(t)) $. Now, $P$ is said to be rigid analytically trivial if there
exists $\Psi\in \GL_{r}(\LL)$ so that $\Psi^{(-1)}=\Phi \Psi $. The existence of such matrix $\Psi$
is equivalent to the isomorphism of the natural map of $\LL$-vector spaces, $\LL \otimes_{\FF_{q}(t)}P^{B}\rightarrow P^{\dag}$, where
\begin{enumerate}
\item[$\bullet$] $P^{\dag}:=\LL \otimes_{\bar{k}(t)}P$, on which $\s$ acts diagonally;

\item[$\bullet$] $P^{B}:=$the $\FF_{q}(t)$-submodule of $P^{\dag}$ fixed by $\s$, the \lq\lq Betti\rq \rq  cohomology of $P$.
\end{enumerate}
We then call that
$\Psi$ is a rigid analytic trivialization for $\Phi$ and in this situation $\Psi^{-1}\mathbf{p}$ is an $\FF_{q}(t)$-basis of $P^{B}$. Note that if $\Psi'$ is also a rigid analytic trivialization for $\Phi$, then one has
\begin{equation}\label{E:Uniq of Psi}
\Psi'^{-1}\Psi\in \GL_{r}(\FF_{q}(t)).
\end{equation}

Given an Anderson $t$-motive $\mathcal{M}$, we obtain a
pre-$t$-motive $M$ by setting
$M:=\bar{k}(t)\otimes_{\bar{k}[t]}\mathcal{M}$ with the $\s$-action
given by $\s(f\otimes m ):=f^{(-1)}\otimes \s m $ for $f\in
\bar{k}(t)$, $m\in \mathcal{M}$. Rigid analytically trivial
pre-$t$-motives that can be constructed from Anderson $t$-motives
using direct sums, subquotients, tensor products, duals and internal
Hom's, are called $t$-motives. These $t$-motives form a neutral
Tannakian category $\mathcal{T}$ over $\FF_{q}(t)$ with fiber functor $M\mapsto M^{B}$.

 For any $t$-motive $M$, let
$\mathcal{T}_{M}$ be the strictly full Tannakian subcategory of
$\mathcal{T}$ generated by $M$, then by Tannakian duality
$\mathcal{T}_{M}$ is equivalent to the category of
finite-dimensional representations over $\FF_{q}(t)$ of an affine
algebraic group scheme $\G_{M}$ defined over $\FF_{q}(t)$. The
$\G_{M}$ is called the motivic Galois group of $M$ and we always have a faithful representation
$$ \varphi:\G_{M}\hookrightarrow \GL(M^{B}). $$
Using difference Galois theory, $\G_{M}$ can be constructed explicitly  as follows.

Let $\Phi\in\GL_{r}(\bar{k}(t))$ represent multiplication by $\s$ on
$M$ and let $\Psi\in {\rm{GL}}_{r}(\LL)$ be a rigid analytic
trivialization for $\Phi$.  Now set $\Psi_{1}$, $\Psi_{2}\in
{\rm{GL}}_{r}(\LL\otimes_{\bar{k}(t)} \LL)$ to be the matrices such
that $(\Psi_{1})_{ij}=\Psi_{ij}\otimes 1$ and
$(\Psi_{2})_{ij}=1\otimes \Psi_{ij}$, and let
$\widetilde{\Psi}:=\Psi_{1}^{-1}\Psi_{2}\in
{\rm{GL}}_{r}(\LL\otimes_{\bar{k}(t)}\LL )$. We define an
$\FF_q(t)$-algebra homomorphism $\mu_{\Psi}: \FF_q(t)[X,1/\det X]
\rightarrow \LL\otimes_{\bar{k}(t)}\LL$ by setting $\mu(X_{ij}) =
\widetilde{\Psi}_{ij}$, where $X=(X_{ij})$ is an $r\times r$ matrix of independent variables.  We let $\Delta_{\Psi}:={\rm{Im}}\hbox{ }
\mu_{\Psi}$ and set
\begin{equation}\label{E:GammaDef}
\Gamma_{\Psi} := {\rm{Spec}}\hbox{ }\Delta_{\Psi}.
\end{equation}
Thus $\Gamma_{\Psi}$ is the smallest closed subscheme of
${\rm{GL}}_{r/\FF_q(t)}$ such that $\widetilde{\Psi} \in
\Gamma_{\Psi}(\LL\otimes_{\bar{k}(t)}\LL)$. We collect the results
developed by Papanikolas as the following.
\begin{theorem}[Papanikolas {\cite{P08}}]
\label{T:GalThy} Let $M$ be a $t$-motive  and let $\Phi\in
{\rm{GL}}_{r}(\bar{k}(t))$ represent multiplication by $\sigma$ on
$M$ and let $\Psi\in {\rm{GL}}_{r}(\LL)$ be a rigid analytic
trivialization for $\Phi$. Then $\G_{\Psi}$ has the following
properties:
\begin{enumerate}
\item $\Gamma_{\Psi}$ is a closed $\FF_q(t)$-subgroup scheme of
${\rm{GL}}_{r/\FF_q(t)}$.
\item $\Gamma_{\Psi}$ is absolutely irreducible and smooth over
$\overline{\FF_{q}(t)}$.
%\item[(c)] $Z_{\Psi}$ is stable under right-multiplication by
%$\Gamma_{\Psi}\times_{\FF_q(t)}\bar{k}(t)$ and is a torsor for
%$\Gamma_{\Psi}\times_{\FF_q(t)}\bar{k}(t)$ over $\bar{k}(t)$.
%\item[(d)] ${\rm{dim}}\hbox{ } \Gamma_{\Psi}= {\rm{tr.deg}}_{\bar{k}(t)} \Lambda_{\Psi}$.
\item $\Gamma_{\Psi}$ is isomorphic to $\Gamma_{M}$ over $\FF_q(t)$.
\item If $\Phi\in \Mat_{r}(\bar{k}[t])\cap \GL_{r}(\bar{k}(t))$ and
${\rm{det}}\Phi=c(t-\th)^{s}$ for some $c\in \bar{k}^{\times}$, then
there exists $\Psi\in \Mat_{r}(\EE)\cap \GL_{r}(\TT)$ so that
$\Psi^{(-1)}=\Phi \Psi$ and one has
$$ {\rm{dim}}\hbox{ }\G_{\Psi}= \trdeg_{\bar{k}}\bar{k}(\Psi(\th)),$$
where $\bar{k}(\Psi(\th))$ is the field generated by all entries of
$\Psi(\th)$ over $\bar{k}$.
\end{enumerate}
\end{theorem}

\begin{remark}\label{rmk for Galois rep} By Theorem \ref{T:GalThy}, we always identify
$\G_{M}$ with $\G_{\Psi}$ in this paper. Using the formalism of Frobenius difference equations, the
faithful representation $\varphi$ can be described as follows. For any $\FF_{q}(t)$-algebra $R$, the map
$\varphi:\G_{\Psi}(R)\rightarrow \GL(R\otimes_{\FF_{q}(t)}M^{B})$ is given by
$$ \gamma\mapsto \left( 1\otimes\Psi^{-1}\mathbf{m}\mapsto (\gamma^{-1}\otimes 1)\cdot (1\otimes\Psi^{-1}\mathbf{m}   )    \right) .$$

\end{remark}
\subsection{Motivic Galois groups of CM Drinfeld modules of rank
2}\label{sec for rank 2 Drin mod} We fix a rank $2$ Drinfeld
$\FF_{q}[t]$-module $\rho$ given by $\rho_{t}=\th+\kappa \t+\t^{2}$,
with $\kappa\in \bar{k}$.  Let
$\Lambda_{\rho}:=\FF_{q}[\th]\omega_{1}+\FF_{q}[\th]\omega_{2}$ be
the period lattice of $\rho$. In this subsection, we suppose that
$\rho$ has complex multiplication, i.e., the fraction field of $
{\rm{End}}(\rho):= \left\{ x\in \CC_{\infty}; \hbox{
}x\Lambda_{\rho}\subseteq \Lambda_{\rho} \right\}$
 is a quadratic field extension of  $k$. Let
\begin{equation}\label{def of Phi rho}
\Phi_{\rho}:=\left(
                         \begin{array}{cc}
                           0 & 1 \\
                           (t-\th) & -\kappa^{(-1)} \\
                         \end{array}
                       \right)\in \Mat_{2}(\bar{k}[t])\cap \GL_{2}(\bar{k}(t))
\end{equation}
define a pre-$t$-motive $M_{\rho}$, i.e., with respect to a fixed
$\bar{k}(t)$-basis $\mathbf{m}$ of the two dimensional $\bar{k}(t)$-vector space
$M_{\rho}$,
 multiplication by $\s$
on $M_{\rho}$ is given by $\Phi_{\rho}$. In fact, $M_{\rho}$ is a $t$-motive and we
review the construction of a rigid analytic trivialization for
$\Phi_{\rho}$ as follows (cf. \cite[$\S 2.5$]{CP08} and \cite[$\S
4.2$]{Pe07}).

Let $\exp_{\rho}(z):=\sum_{i=0}^{\infty}\alpha_{i}z^{q^{i}}$ be
the exponential function of $\rho$. Given $u\in \CC_{\infty}$ we
consider the Anderson generating function
\begin{equation}\label{E:AndGF}
f_{u}(t):=\sum_{i=0}^{\infty} \exp_{\rho} \biggl(
\frac{u}{\theta^{i+1}} \biggr)t^{i} =
\sum_{i=0}^{\infty}\frac{\alpha_{i} u^{q^{i}} }{ \theta^{q^{i}}-t }
\in \TT
\end{equation}
and note that $f_{u}(t)$ is a meromorphic function on
$\CC_{\infty}$.  It has simple poles at $\theta$, $\theta^{q},
\ldots$ with residues $-u$, $-\alpha_{1} u^{q}, \ldots$
respectively.  Using $\rho_{t}(\exp_{\rho} (\frac{u}{\theta^{i+1}}))
= \exp_{\rho}(\frac{u}{\theta^{i}})$, we have
\begin{equation}\label{E:fu1}
\kappa f_{u}^{(1)}+f_{u}^{(2)} =(t-\theta)f_{u} + \exp_{\rho}(u).
\end{equation}
Since $f_{u}^{(m)}(t)$ converges away from $\{
\theta^{q^{m}},\theta^{q^{m+1}},\ldots \}$ and
$\Res_{t=\theta}f_{u}(t)=-u$, we have
\begin{equation}\label{E:fu2}
\kappa f_{u}^{(1)}(\theta)+f_{u}^{(2)}(\theta) =-u+\exp_{\rho}(u)
\end{equation}
by specializing \eqref{E:fu1} at $t=\theta$. Moreover, using the
difference equation $F_{\t}(\th z)-\th
F_{\t}(z)=\hbox{exp}_{\rho}(z)^{q}$ one has
\begin{equation}\label{f u (1) theta}
f_{u}^{(1)}(\th)=F_{\t}(u).
\end{equation}

 Recall the analogue
of the Legendre relation proved by Anderson,
\begin{equation}\label{E:Legendre}
\omega_{1} F_{\tau}(\omega_{2})- \omega_{2}
F_{\tau}(\omega_{1})=\tilde{\pi}/ \xi,
\end{equation}
where $\xi\in \overline{\FF_q}^{\times}$ satisfies $\xi^{(-1)}=-\xi$
and $\tilde{\pi}$ is a generator of the period lattice of the
Carlitz module $C$.  We pick a suitable choice
$(-\theta)^{\frac{1}{q-1}}$ of the $(q-1)$-st root of $-\theta$ so
that $\Omega(\theta)=\frac{-1}{\tilde{\pi}}$, where
\[
\Omega(t):=(-\theta)^{\frac{-q}{q-1}} \prod_{i=1}^{\infty} \biggl(
1-\frac{t}{\theta^{q^{i}}} \biggr) \in \EE,
\]
(see \cite[Cor.~5.1.4]{ABP04}). Note that $\Omega$ satisfies the difference equation $\Omega^{(-1)}=(t-\th)\Omega$.   Now put $f_{i}:=f_{\omega_{i}}(t)$
for $i=1,2$, and define
\begin{equation}\label{def of Psi rho}
\Psi_{\rho}:=\xi \Omega \left(
                                \begin{array}{cc}
                               -f_{2}^{(1)}    &  f_{1}^{(1)} \\
                                  \kappa f_{2}^{(1)}+f_{2}^{(2)}   &  -\kappa f_{1}^{(1)}-f_{1}^{(2)} \\
                                \end{array}
                              \right).
\end{equation}
Then $\Psi_{\rho}\in \hbox{GL}_{2}(\TT)\cap \hbox{Mat}_{2}(\mathbb{E})$ and
by \eqref{E:fu1} we have $\Psi_{\rho}^{(-1)} = \Phi_{\rho}\Psi_{\rho}$. By specializing
$\Psi_{\rho}$ at $t=\theta$ and using the analogue of the Legendre
relation, we obtain
\begin{equation}\label{Psi th}
 \left(\Psi_{\rho}^{-1}\right)^{\rm{tr}}(\th)=\left(%
\begin{array}{cc}
  \o_{1} &\o_2  \\
   -F_{\t}(\o_1) & -F_{\t}(\o_{2}) \\
\end{array}%
\right)  .
\end{equation}
Therefore by Theorem~\ref{T:GalThy} we have the following
equivalence
$$
 {\rm{dim}}\hbox{ } \Gamma_{\Psi_{\rho}}=2 \Leftrightarrow
\trdeg_{\bar{k}} \bar{k}( \omega_{1}, \omega_{2},
F_{\tau}(\omega_{1}), F_{\tau}(\omega_{2}))=2.
$$
As $\rho$ has complex multiplication, Thiery \cite{Thi92} has shown
that $$\trdeg_{\bar{k}} \bar{k}( \omega_{1}, \omega_{2},
F_{\tau}(\omega_{1}), F_{\tau}(\omega_{2}))=2,$$ whence
${\rm{dim}}\hbox{ }\G_{\Psi_{\rho}}=2$.
\begin{remark}\label{rmk for pi omega}
One can  prove  ${\rm{dim}}\hbox{ }\G_{\Psi_{\rho}}=2$ without using
Thiery's result, but we do not discuss the details here. We further note that $\left \{ \o_{1},F_{\t}(\o_{1})   \right\}$
and $\left\{ \tilde{\pi},\o_{1} \right\}$ are transcendental bases of $\bar{k}( \omega_{1}, \omega_{2},
F_{\tau}(\omega_{1}), F_{\tau}(\omega_{2}))$ over $\bar{k}$.
\end{remark}

\begin{lemma}\label{rest of scalars} Let $\rho$, $M_{\rho}$ and $\Psi_{\rho}$ be defined as above.
Suppose that the CM field of $\rho$ is separable over $k$, then the motivic Galois group $\G_{\Psi_{\rho}}$ is a non-split torus over $\FF_{q}(t)$.
\end{lemma}
\begin{proof}
 Let $\mathcal{K}:={\rm{End}}_{\mathcal{T}}(M_{\rho})$ be the ring consisting of all left
 $\bar{k}(t)[\s,\s^{-1}]$-module endomorphisms of $M_{\rho}$. Since the representation $\varphi:\G_{\Psi_{\rho}}\rightarrow \GL(M_{\rho}^{B})$ is functorial in $M_{\rho}$ (cf. \cite[Thm. 4.5.3]{P08}), for any
$\FF_{q}(t)$-algebra $R$, one has the
embedding
\begin{equation}\label{embedding into res of scalars}
 \G_{\Psi_{\rho}}(R) \hookrightarrow
\hbox{Cent}_{{\rm{GL}}_{2}(R)}((\mathcal{K}\otimes_{\FF_{q}(t)}R)^{\times})\hbox{ (cf. \cite[p. 11]{CPY09})}.
\end{equation}
One also has the property that $\mathcal{K}$ can be identified with
the CM field of $\rho$ (cf. \cite[Prop. 2.4.3, Remark 2.4.4]{CP08})
and hence $\mathcal{K}$ is a separable quadratic extension of
$\FF_{q}(t)$. Thus, the right hand side of (\ref{embedding into res
of scalars}) consists of the $R$-valued points of the restriction of
scalars ${\rm{Res}}_{\mathcal{K}/ \FF_{q}(t)} ({\GG_{m}}_{/
\mathcal{K}} )$. Since $ {\rm{dim}}\hbox{ }\G_{\Psi_{\rho}}=2$, it follows
that $\G_{\Psi_{\rho}}\cong {\rm{Res}}_{\mathcal{K}/ \FF_{q}(t)
}({\GG_{m}}_{/ \mathcal{K}} )$.
\end{proof}

We close this section with the following proposition which will be
used in the proof of Theorem \ref{thm 2 intro}.

\begin{proposition}\label{prop for F ij}
Let $\rho$, $M_{\rho}$ and $\mathbf{m}$ be defined as above. Given nonzero $f\in {\rm{End}}_{\mathcal{T}}(M_{\rho})$,  let $F=(F_{ij})\in \GL_{2}(\bar{k}(t))$ satisfy
$f(\mathbf{m})=F \mathbf{m}$. Then we have:
\begin{enumerate}
\item All the entries $F_{ij}$ are regular at $t=\th,\th^{q},\th^{q^{2}},\ldots$.

\item $F_{21}(\th)=0$.

\item $F_{11}(\th)$ lies in the CM field of $\rho$.

\end{enumerate}
\end{proposition}
\begin{proof}
Actually, the denominator of each $F_{ij}$ is shown to be in $\FF_{q}[t]$ (see \cite[p. 146]{P08}), whence the result of (1).

For the proof of (2), we first note that since $f$ is a left $\bar{k}(t)[\s,\s^{-1}]$-module homomorphism, we have $F^{(-1)}\Phi_{\rho}=\Phi_{\rho} F$. Therefore we have the difference equation $(F\Psi_{\rho})^{(-1)}=\Phi_{\rho}(F \Psi_{\rho})$, whence
\begin{equation}\label{E:F eta}
F\Psi_{\rho}=\Psi_{\rho} \eta
\end{equation}
 for some $\eta\in \GL_{2}(\FF_{q}(t))$ because of (\ref{E:Uniq of Psi}). For any $x,y\in \CC_{\infty}$ with $y\neq 0$, we denote by $x\sim y$
if $x/y \in \bar{k}$. By specializing (\ref{E:F eta}) at $t=\th$, we obtain that $F_{21}(\th)\sim \o_{1}^{2}/ \tilde{\pi}$ since $\o_{1}$ and $\o_{2}$ are linearly dependent over $\rm{End}(\rho)$. Since $F_{21}(\th)\in \bar{k}$ and
$\left\{ \tilde{\pi},\o_{1}  \right\}$ is an algebraically independent set over $\bar{k}$ by Remark \ref{rmk for pi omega}, $F_{21}(\th)$ has to be zero.

 The property (3) is a  consequence of (2).  In fact, the matrix $\eta=\Psi_{\rho}^{-1}F\Psi_{\rho}$ is the image of $f$ in ${\rm{End}}(M_{\rho}^{B})=\Mat_{2}(\FF_{q}(t))$ (cf. \cite[$\S 3.2$]{CP08}). By taking the conjugation by $\Psi_{\rho}$, the image of ${\rm{End}}(M_{\rho})$ in
$\Mat_{2}(\FF_{q}(t))$ is a quadratic field extension of $\FF_{q}(t)$ that is isomorphic to the fraction field of ${\rm{End}}(\rho)$. As
$F_{21}(\th)=0$ and $$  \Psi_{\rho}(\th)^{-1} F(\th)^{-1} =\eta(\th)^{-1} \Psi_{\rho}(\th)^{-1},$$
we see that $F_{11}(\th)^{-1}$ is an eigenvalue of the matrix $\eta(\th)^{-1}$ with eigenvector as the first column of $\Psi_{\rho}(\th)^{-1}$, whence we prove (3).

\end{proof}

\section{Periods of third kind for rank 2 Drinfeld modules}

\subsection{Biderivations and extensions of Drinfeld modules}

 Fix any
rank $2$ Drinfeld $\FF_{q}[t]$-module $\rho$ given by
$\rho_{t}:=\th+\kappa \t+ \Delta \t^{2}$ with $\kappa, \Delta\in
\bar{k}$, $\Delta \neq 0$, and let $C$ be the Carlitz module. We
also fix generators $\{ \o_1,\o_2 \}$ of the period lattice
$\L_{\rho}$ over $\FF_{q}[\th]$. We are interested in the two
dimensional $t$-module $\phi$ which is an extension of $\rho$ by $C$
and whose Lie algebra is split, i.e., $\phi$ fits into a short exact
sequence of $\FF_{q}[t]$-modules
$$ 0\rightarrow C\rightarrow \phi \twoheadrightarrow \rho\rightarrow 0    $$
and $\partial \phi_{t}=\th I_{2}$. Let $\Ext^{1}_{0}(\rho,C)$ be the
category of the two dimensional $t$-modules $\phi$ as above, then it
forms a group under Bare sum up to Yoneda equivalence. However,
$\Ext^{1}_{0}(\rho,C)$ can be  described explicitly as follows (see
\cite[$\S2$]{PR03}).

Let ${\rm{Der}}_{0}(\rho,C)$ be the $\FF_q$-vector space consisting
of all $\FF_q$-linear maps $\d:  \FF_{q}[t]\rightarrow
\CC_{\infty}[\t]\t$ satisfying
$$       \d_{ab}=C_{a}\delta_{b}+\delta_{a}\rho_{b} \hbox{ for
}a,b\in \FF_{q}[t]       .$$ Any $\d\in {\rm{Der}}_{0}(\rho,C)$ is
called a $(\rho,C)$-biderivation and it is uniquely determined by
the image $\d_{t}=:h \in \CC_{\infty}[\t]\t $. Hence any element
$h\in \CC_{\infty}[\t]\t$ defines an element $\d_{h}\in
{\rm{Der}}_{0}(\rho,C)$ given by $\d_{h}:t\mapsto h$. Thus, we have
a canonical isomorphism of $\FF_{q}$-vector spaces
$$h\mapsto \d_{h}:\CC_{\infty}[\t]\t \rightarrow {\rm{Der}}_{0}(\rho,C).    $$

An element $\d\in {\rm{Der}}_{0}(\rho,C)$ is called inner if there
exists $U\in \CC_{\infty}[\t]$ so that
$$\d_{a}=\d_{a}^{(U)}:= U \rho_{a}- C_{a}U\hbox{ for all }a\in \FF_{q}[t] .$$
The $\FF_{q}$-vector space of all such inner
$(\rho,C)$-biderivations is denoted by ${\rm{Der_{inn}}}(\rho,C  )
$.

Given any $\d\in {\rm{Der}}_{0}(\rho,C)$, it defines a $t$-module
$\phi\in \Ext^{1}_{0}(\rho,C)$
  in the following way
$$\phi_{t}:=\left(
              \begin{array}{cc}
                \rho_{t} & 0 \\
                \d_{t} & C_{t} \\
              \end{array}
            \right)\in {\rm{Mat}}_{2}(\CC_{\infty}[\t])
   .$$ Conversely, one observes that every $\phi\in \Ext^{1}_{0}(\rho,C)$ defines a
   unique $\d \in {\rm{Der}}_{0}(\rho,C)$. Note that if $\d^{(U)}$ is an inner derivation, then the $t$-module
$\phi$ associated to $\d^{(U)}$ is split. In this case, the matrix
$\gamma:=\left(
                         \begin{array}{cc}
                          1  & 0 \\
                           U & 1 \\
                         \end{array}
                       \right)
 $ provides the splitting:
 $$\gamma^{-1}\phi_{a} \gamma=\left(
                                    \begin{array}{cc}
                                      \rho_{a} & 0 \\
                                      0 & C_{a} \\
                                    \end{array}
                                  \right)
  \hbox{ for all }a\in \FF_{q}[t] .$$ However, every split
  extension in $\Ext^{1}_{0}(\rho,C)$ arises in this way.

 Given two extensions in $\Ext^{1}_{0}(\rho,C)$ which are Yoneda
  equivalent, by the definition of the Yoneda equivalence the
  corresponding elements in ${\rm{Der}}_{0}(\rho,C)$ differ by an inner
  derivation. Since  the (Bare) sum on
  ${\rm{Ext}}^{1}_{0}(\rho,C)$ corresponds to the usual addition on
  ${\rm{Der}}_{0}(\rho,C)$, we see that $\Ext^{1}_{0}(\rho,C)$ is
  isomorphic to ${\rm{Der}}_{0}(\rho,C)/ {\rm{Der_{inn}}}(\rho,C)$
  as $\FF_{q}$-vector spaces and hence the space
     $\Ext_{0}^{1}(\rho,C)$ is in bijection with $\left\{ \alpha \t;\hbox{ }\alpha\in \CC_{\infty}
     \right\}$. Moreover,  $\Ext_{0}^{1}(\rho,C)$ has a $t$-module structure
isomorphic to the rank $2$ Drinfeld $\FF_{q}[t]$-module $\rho$.

\subsection{Periods of third kind for $\rho$}\label{subsec for 3rd
kind}

Given $\alpha\in \bar{k}^{\times}$, we let $\d\in
\hbox{Der}_{0}(\rho,C)$ be defined by $\d_{t}=\a \t$, and let $\phi$
be its corresponding $t$-module given by
$$\phi_{t}:=\left(
              \begin{array}{cc}
                \rho_{t} & 0 \\
                \d_{t} & C_{t} \\
              \end{array}
            \right).
   $$ Then its exponential function can be written as

\begin{equation}\label{exp phi}
    {\rm{exp}}_{\phi}\left(
                         \begin{array}{c}
                           z_{1} \\
                           z_{2} \\
                         \end{array}
                       \right)=\left(
                                 \begin{array}{c}
                                   {\rm{exp}}_{\rho}(z_{1}) \\
                                   {\rm{exp}}_{C}(z_{2})+\mathcal{G}_{\d}(z_{1}) \\
                                 \end{array}
                               \right)\hbox{ (cf. \cite[p.422]{PR03})},
     \end{equation}
where $\mathcal{G}_{\d}$ is an $\FF_{q}$-linear, entire function on
     $\CC_{\infty}$ satisfying the properties:
\begin{enumerate}
\item[$\bullet$] $\mathcal{G}_{\d}(z)\equiv 0\hbox{ }({\rm{mod}}\hbox{ }z^{q});$

\item[$\bullet$] $\mathcal{G}_{\d}(a(\th)z )=C_{a} (\mathcal{G}_{\d}(z))+\d_{a}({\rm{exp}}_{\rho}(z))\hbox{ }\forall a\in
\FF_{q}[t]$.
\end{enumerate}
Thus, $\phi$ is uniformizable and its period lattice is given by
\begin{equation}\label{Lambda phi}
 \L_{\phi}=\FF_{q}[\th]\hbox{-} {\rm{Span}}\left\{ \left(
                       \begin{array}{c}
                         \o_{1} \\
                         \l_{1} \\
                       \end{array}
                     \right), \left(
                       \begin{array}{c}
                         \o_{2} \\
                         \l_{2} \\
                       \end{array}
                     \right), \left(
                       \begin{array}{c}
                         0 \\
                         \tilde{\pi} \\
                       \end{array}
                     \right)
     \right\}
\end{equation} for some $\l_{1},\l_{2}\in \CC_{\infty}$. For any $\left(
                                                \begin{array}{c}
                                                  \omega \\
                                                  \lambda \\
                                                \end{array}
                                              \right)\in \Lambda_{\phi}
     $, we call $\lambda$ a period of third kind for $\rho$
     associated to $\a \t$. In the following subsection, we give a
     detailed proof of the formula for such $\l$ as in Theorem \ref{thm 1
     intro}.

\subsection{Proof of Theorem \ref{thm 1 intro}}

Recall that $\rho_{t}= \th+\kappa \t+ \Delta \t^{2}$ for $\kappa,
\Delta\in \bar{k}$, $\Delta \neq 0$. First, we assume that
$\Delta=1$. Since the restriction of $\s^{-1}$ to $\CC_{\infty}$ is
essentially the same as the action of $\t$ on $\CC_{\infty}$, we
extend the action of $\t$ to $\CC_{\infty}((t))$ by identifying it
with $\s^{-1}$ (cf. \S \ref{sec Frob Twist}). For $j=1,2,$ we define
$$\left(
        \begin{array}{c}
          f_{j}\\
          g_{j} \\
        \end{array}
      \right):=\sum_{i=0}^{\infty} {\rm{exp}}_{\phi}\left( (\partial \phi_{t})^{-i-1} \left(
                       \begin{array}{c}
                         \o_{j} \\
                         \l_{j} \\
                       \end{array}
                     \right) \right)t^{i}\in {\rm{Mat}}_{2\times
                     1}(\CC_{\infty}[[t]]),
   $$ where $\l_{j}$ is given in (\ref{Lambda phi}).   Then using the functional equation $  \exp_{\phi}(\partial \phi_{t}(\mathbf{z}) ) =\phi_{t}\left(  \exp_{\phi}(\mathbf{z}) \right) $  one has
   $$ \phi_{t} \left(
                       \begin{array}{c}
                         f_{j} \\
                         g_{j} \\
                       \end{array}
                     \right)=   \left(
                       \begin{array}{c}
                        \th f_{j}+\kappa f_{j}^{(1)}+f_{j}^{(2)}  \\
                        \a  f_{j}^{(1)}+\th g_{j}+g_{j}^{(1)}  \\
                       \end{array}
                     \right)  =t \left(
                       \begin{array}{c}
                         f_{j} \\
                         g_{j} \\
                       \end{array}
                     \right) . $$
Hence, for $j=1,2,$ one has
$$
    \begin{array}{l}
  \k {f_{j}}^{(1)} + {f_{j}}^{(2)}=(t-\th)f_{j},    \\
  \a  {f_{j}}^{(1)}+{g_{j}}^{(1)}=(t-\th)g_{j}.   \\
    \end{array}
  $$
This leads to the following difference equation
$$ \left(
     \begin{array}{ccc}
      f_{1}  & f_{2} & 0 \\
      f_{1}^{(1)}  & f_{2}^{(1)} & 0 \\
      g_{1}  & g_{2} & \frac{1}{\Omega^{(-1)} } \\
     \end{array}
   \right)^{(1)}=\left(
                   \begin{array}{ccc}
                    0  & 1 & 0 \\
                    (t-\th) &-\k  & 0 \\
                    0  & -\a & (t-\th) \\
                   \end{array}
                 \right)\left(
     \begin{array}{ccc}
      f_{1}  & f_{2} & 0 \\
      f_{1}^{(1)}  & f_{2}^{(1)} & 0 \\
      g_{1}  & g_{2} & \frac{1}{\Omega^{(-1)} } \\
     \end{array}
   \right).
   $$
Put
$$\Phi:=\left(
                   \begin{array}{ccc}
                    \k  & 1 & 0 \\
                    (t-\th) & 0  & 0 \\
                    \a  & 0 & 1 \\
                   \end{array}
                 \right)\hbox{ and }\Psi:=  \left(
     \begin{array}{ccc}
     \Omega {f_{1}}^{(1)}  & \Omega {f_{2}}^{(1)} & 0 \\
     \Omega {f_{1}}^{(2)}  &\Omega {f_{2}}^{(2)} & 0 \\
     \Omega {g_{1}}^{(1)}  &\Omega {g_{2}}^{(1)} & 1 \\
     \end{array}
   \right),$$ then we have $$\Psi^{(-1)}=\Phi \Psi  .$$

For any vector $\mathbf{v}=(v_{1},v_{2})^{tr}\in \CC_{\infty}^{2}$
we put $\| \mathbf{v} \|:=\hbox{max}\{
|v_{1}|_{\infty},|v_{2}|_{\infty} \} $. Then we have $ \| \mathbf{v}
+\mathbf{w} \|\leq \hbox{max}\{  \| \mathbf{v} \|,\| \mathbf{w} \|
\} $ for $\mathbf{v},\mathbf{w}\in \CC_{\infty}^{2}  $. We claim
that for each $j=1,2$,
\begin{enumerate}
\item[$\bullet$] $\left(%
\begin{array}{c}
  f_{j} \\
  g_{j} \\
\end{array}%
\right)$ converges on $|t |_{\infty} < |\th|_{\infty}=q $;

\item[$\bullet$] $\hbox{Res}_{t=\th}f_{j}=-\o_{j}$, $\hbox{Res}_{t=\th} g_{j}=-\l_{j}
$.

\end{enumerate}

 We write
 $$  {\rm{exp}}_{\phi}  \left(
                         \begin{array}{c}
                           z_{1} \\
                           z_{2} \\
                         \end{array}
                       \right)=\sum_{i=0}^{\infty} \a_{i}  \left(
                                                           \begin{array}{c}
                                                             z_{1}^{q^{i}} \\
                                                             z_{2}^{q^{i}} \\
                                                           \end{array}
                                                         \right),
 $$ where $$\a_{0}=\left(
                    \begin{array}{cc}
                      1 & 0 \\
                      0 & 1 \\
                    \end{array}
                  \right) \hbox{ and }\a_{i}\in
                  {\rm{Mat}}_{2}(\CC_{\infty})\hbox{ for all }i\in\NN.
$$
For each $1\leq \ell\leq 2$,  by definition we have
$$\left(%
\begin{array}{c}
  f_{\ell} \\
  g_{\ell} \\
\end{array}%
\right)=\sum_{j=0}^{\infty} \sum_{i=0}^{\infty} \a_{i} \left(
\frac{1}{\th^{ q^{i}(j+1)}} \left(%
\begin{array}{c}
  \o_{\ell}^{q^{i}} \\
  \l_{\ell}^{q^{i}} \\
\end{array}%
\right)  \right)t^{j}
$$
For any $|t|_{\infty}< |\th|_{\infty}$, since $\exp_{\phi}$ is
entire on $\CC_{\infty}^{2}$, there exists a positive integer $N$
sufficiently large so that
$$
\left \|  \sum_{i=0}^{\infty} \a_{i} \left(
\frac{1}{\th^{ q^{i}(j+1)}} \left(%
\begin{array}{c}
  \o_{\ell}^{q^{i}} \\
  \l_{\ell}^{q^{i}} \\
\end{array}%
\right)  \right) \right \| |t|_{\infty}^{j} \leq  \hbox{max}_{0\leq i \leq N} \left\{ \left \| \a_{i} \left(%
\begin{array}{c}
  \o_{\ell}^{q^{i}} \\
  \l_{\ell}^{q^{i}} \\
\end{array}%
\right) \right  \| \frac{1 }{ |\th|_{\infty}^{q^{i}(j+1)} } \right\}
|t|_{\infty}^{j},
$$
which converges to $0$ as $j \rightarrow
\infty$. Hence $\left(%
\begin{array}{c}
  f_{\ell} \\
  g_{\ell} \\
\end{array}%
\right)$ converges on $ |t|_{\infty}< |\th|_{\infty}$ for
$\ell=1,2$.

For the second part of the above claim, it suffices to prove that
for each $1\leq \ell\leq 2$,

$$ \left(%
\begin{array}{c}
  f_{\ell} \\
  g_{\ell} \\
\end{array}%
\right)-\left(%
\begin{array}{c}
  -\o_{\ell}/ (t-\th) \\
  -\l_{\ell}/ (t-\th) \\
\end{array}%
\right)=\left(%
\begin{array}{c}
  f_{\ell} \\
  g_{\ell} \\
\end{array}%
\right)-\sum_{j=0}^{\infty} \frac{1}{\th^{j+1} } \left(%
\begin{array}{c}
  \o_{\ell} \\
  \l_{\ell} \\
\end{array}%
\right)t^{j}=  \sum_{j=0}^{\infty} \sum_{i=1}^{\infty} \a_{i} \left(
\frac{1}{\th^{ q^{i}(j+1)}} \left(%
\begin{array}{c}
  \o_{\ell}^{q^{i}} \\
  \l_{\ell}^{q^{i}} \\
\end{array}%
\right)  \right)t^{j} $$ converges at $t=\th$. The proof follows
from the argument of approximation as above.

For each $j=1,2$,  note that by (\ref{exp phi}) we have
$f_{j}=f_{\o_{j}}$, where $f_{\o_{j}}$ is the Anderson generating
function of $\o_{j}$ (cf. (\ref{E:AndGF})). Since
$f_{j}^{(1)}(\th)=F_{\t}(\o_{j})$ for $j=1,2,$ (cf. (\ref{f u (1)
theta})),
 we have
$$ \Psi(\th)=\left(
               \begin{array}{ccc}
                -F_{\t}(\o_{1})/ \tilde{\pi}  &  -F_{\t}(\o_{2})/ \tilde{\pi}  & 0 \\
               (\o_{1}+\k F_{\t}(\o_{1}) )/ \tilde{\pi}   & (\o_{2}+\k F_{\t}(\o_{2}))/ \tilde{\pi}   & 0 \\
               (\l_{1}+\alpha F_{\t}(\o_{1}))/\tilde{\pi}   & (\l_{2}+\alpha F_{\t}(\o_{2}))/ \tilde{\pi} & 1 \\
               \end{array}
             \right).
    $$

 Let $\xi$ be given as in (\ref{E:Legendre}). Define
$$A:= \left(
        \begin{array}{ccc}
         \xi  & 0 & 0 \\
        -\xi \kappa   & -\xi & 0 \\
         0  & 0 & 1 \\
        \end{array}
      \right)
   \hbox{ and }\mathbf{\Phi}:=\left(
                          \begin{array}{ccc}
                            0 & 1 & 0 \\
                            (t-\th) & -\kappa^{(-1)} & 0 \\
                            \a/\xi & 0 &  1\\
                          \end{array}
                        \right), $$ then one has
\begin{equation}\label{relation for A sigma A}
   A^{(-1)} \Phi= \mathbf{\Phi} A  .
\end{equation}
Let $f_{u}$ be the Anderson generating function of $u$ (cf.
(\ref{E:AndGF})), we put  $$\left( G_{1},G_{2}\right):=
\left(-(t-\th)f_{u}-\a/\xi,-f_{u}^{(1)}  \right) \left(
                                                                                            \begin{array}{cc}
                                                                                           -\xi \Omega f_{2}^{(1)}     & \xi \Omega f_{1}^{(1)} \\
                                                                                            \xi (t-\th)\Omega f_{2}    &-\xi (t-\th) \Omega f_{1}   \\
                                                                                            \end{array}
                                                                                          \right).
     $$ and define
     $$  \mathbf{\Psi}:=\left(
                   \begin{array}{ccc}
                    -\xi \Omega f_{2}^{(1)}  & \xi \Omega f_{1}^{(1)} & 0 \\
                    \xi (t-\th)\Omega f_{2} & -\xi (t-\th) \Omega f_{1} &  0\\
                    G_{1}  & G_{2} & 1 \\
                   \end{array}
                 \right), $$
then using (\ref{E:fu1}) one has $$\mathbf{\Psi}^{(-1)} =
\mathbf{\Phi} \mathbf{\Psi} . $$ Since $\mathbf{\Psi}\in
\GL_{3}(\TT)$ (cf. $\S$\ref{sec for rank 2 Drin mod}) and $\det
\Psi=\frac{-1}{\xi^{2}}
 \det \mathbf{\Psi}$, by \cite[Prop. 3.1.3]{ABP04} both of
 $\mathbf{\Psi}$  and $\Psi$ are in $\GL_{3}(\TT)\cap
 \Mat_{3}(\mathbb{E})$. Using (\ref{relation for
A sigma A}) we have $ ( \mathbf{\Psi}^{-1}A \Psi )^{(-1)}=
\mathbf{\Psi}^{-1}A \Psi $. As $\mathbf{\Psi}^{-1}A\Psi\in
\GL_{3}(\TT)$, we see that
\begin{equation}\label{relation for Psi and bf Psi}
A \Psi=  \mathbf{\Psi} \gamma
\end{equation}
for some $\gamma \in {\rm{GL}}_{3}(\FF_{q}[t]) $, because the
subring of $\TT$ fixed by $\s$ is $\FF_{q}[t]$. According to (\ref{f
u (1) theta}), by specializing (\ref{relation for Psi and bf Psi})
at $t=\th$ it follows that $\gamma$ is of the form
$$\gamma=\left(
           \begin{array}{ccc}
             0 &-1  & 0\\
            1  & 0 & 0 \\
            a  & b & 1 \\
           \end{array}
         \right)
  $$ for some $a,b\in \FF_{q}[t]$, whence
\begin{equation}\label{formulae for u}
    \begin{array}{c}
     \l_{1}= -\xi \left( u F_{\t}(\o_{1}) - \o_{1}F_{\t}(u)\right) + a(\th)\tilde{\pi}    \\
       \l_{2}=-\xi \left( u F_{\t}(\o_{2}) - \o_{2}F_{\t}(u) \right) + b(\th)\tilde{\pi}    . \\
    \end{array}
\end{equation}

For the general case that $\Delta\neq 1$, we let $\epsilon\in \bar{k}^{\times}$ be a $(q^{2}-1)$st root of $1/ \Delta$ satisfying $\epsilon^{q+1}/ \xi= 1/\sqrt[q-1]{-\Delta}$. We define $\nu$ to be the
rank $2$ Drinfeld $\FF_{q}[t]$-module given by
$\nu_{t}:=\epsilon^{-1} \rho_{t}
 \epsilon$, whose leading coefficient in $\t$ is $1$. Then we have
\begin{equation}\label{exp nu F nu}
 {\rm{exp}}_{\nu}=\epsilon^{-1}\circ
{\rm{exp}}_{\rho}\circ \epsilon, \hbox{  }
\mathcal{F}_{\t}=\epsilon^{-q}\circ F_{\t} \circ \epsilon,
\end{equation}where $F_{\t}$ (resp. $\mathcal{F}_{\t}$) is the quasi-periodic
function of $\rho$ (resp. $\nu$) associated to $\t$. We define
$\varphi$ to be the two dimensional $t$-module given by
$$ \varphi_{t}:= \left(
                       \begin{array}{cc}
                         \epsilon^{-1} & 0 \\
                         0 & 1 \\
                       \end{array}
                     \right) \left(
                               \begin{array}{cc}
                                 \rho_{t} & 0 \\
                                 \alpha \t & C_{t} \\
                               \end{array}
                             \right)\left(
                                      \begin{array}{cc}
                                        \epsilon & 0 \\
                                        0 & 1 \\
                                      \end{array}
                                    \right)=\left(
                                              \begin{array}{cc}
                                                \nu_{t} & 0 \\
                                                \alpha \epsilon^{q}\t & C_{t} \\
                                              \end{array}
                                            \right),
    $$ then we have
\begin{equation}\label{exp varphi}
    {\rm{exp}}_{\varphi}=\left(
                                     \begin{array}{cc}
                                       \epsilon^{-1} & 0 \\
                                       0 & 1 \\
                                     \end{array}
                                   \right)\circ {\rm{exp}}_{\phi} \circ\left(
                                                             \begin{array}{cc}
                                                               \epsilon & 0 \\
                                                               0 & 1 \\
                                                             \end{array}
                                                           \right)
      .
\end{equation}

Put $\bar{\o}_{i}:=\o_{i}/ \epsilon$ for $i=1,2,$ and $\bar{u}:=u/
\epsilon$. Let $\L_{\nu}$ be the period lattice of $\nu$, then we
have
       $\L_{\nu}:=\FF_{q}[\th]\hbox{-Span}\left\{\bar{\o}_1,\bar{\o}_2
      \right\}$ and  $\exp_{\nu}(\bar{u})= \a \epsilon^{q}/ \xi
      $. Since the leading coefficient of $\nu_{t}$ is $1$,  by (\ref{formulae for
      u}) and (\ref{Lambda phi}) the period lattice $\L_{\varphi}:=\hbox{Ker exp}_{\varphi}$ is the
      $\FF_{q}[\th]$-module generated by
      $$ \left\{ \left(
                   \begin{array}{c}
                     \bar{\o}_1 \\
                     \bar{\l}_1 \\
                   \end{array}
                 \right),\left(
                           \begin{array}{c}
                             \bar{\o}_2 \\
                             \bar{\l}_2 \\
                           \end{array}
                         \right),\left(
                                   \begin{array}{c}
                                     0 \\
                                     \tilde{\pi} \\
                                   \end{array}
                                 \right)
           \right\} ,$$ where $$ \bar{\l}_i=-\xi\left(\bar{u} \mathcal{F}_{\t}( \bar{\o}_i )-\bar{\o}_i \mathcal{F}_{\t}(\bar{u})\right) \hbox{ for }i=1,2.    $$
Now, given any $\left(
                  \begin{array}{c}
                    \o \\
                    \l \\
                  \end{array}
                \right)\in \hbox{Ker exp}_{\phi}
$, by (\ref{exp varphi}) we see that $\left(
                                        \begin{array}{c}
                                          \o/ \xi \\
                                          \l \\
                                        \end{array}
                                      \right)\in \L_{\varphi}
$. Thus, we derive that $\l=-\xi \epsilon^{-q-1}\left(u
F_{\t}(\o)-\o F_{\t}(u) \right)+f(\th)\tilde{\pi} $ for some $f\in
\FF_{q}[t]$ by using $\mathcal{F}_{\t}=\epsilon^{-q}\circ F_{\t}
\circ \epsilon$ and the property that $F_{\t}|_{\L_{\rho}}:
\L_{\rho}\rightarrow \CC_{\infty}$ is $\FF_{q}[\th]$-linear.

\section{Algebraic independence of Drinfeld logarithms}

\subsection{The setting and the reduction of Theorem \ref{thm 2 intro}}\label{sec main thm}

 In this subsection, we fix a rank $2$ Drinfeld
$\FF_{q}[t]$-module $\rho$  given by $\rho_{t}:=\th+\kappa
\tau+\tau^{2}$, with $\kappa\in \bar{k}$, and assume that $\rho$ has
complex multiplication.

 Let $\Phi_{\rho}$ be defined in
(\ref{def of Phi rho}) and $\Psi_{\rho}$ be defined in
(\ref{def of Psi rho}). Let $M_{\rho}$ be the $t$-motive
associated to $\rho$ defined by $\Phi_{\rho}$ (cf. $\S$\ref{sec for rank 2
Drin mod}) and let $\G_{\Psi_{\rho}}$ be its motivic Galois
group. Given $u_{1},\dots,u_{n}\in \CC_{\infty}$ with
$\exp_{\rho}(u_{i}) =: \alpha_{i}\in \bar{k}$ for $i=1, \dots, n$.
 For each $1\leq i \leq n$, we let $f_{u_{i}}$ be the Anderson
generating function of $u_i$ (cf. (\ref{E:AndGF})) and put
$$
\mathbf{g}_{i}:=\left(
                  \begin{array}{c}
                    g_{i1} \\
                    g_{i2} \\
                  \end{array}
                \right):=\left(
                           \begin{array}{c}
                           -\kappa f_{u_{i}}^{(1)}-f_{u_{i}}^{(2)}   \\
                            -f_{u_{i}}^{(1)}  \\
                           \end{array}
                         \right)
\hbox{ and } \mathbf{h}_{i}:=\left(
                               \begin{array}{c}
                                 \a_i \\
                                 0 \\
                               \end{array}
                             \right).
$$
We define
$$
\Phi_{i}:=\left(
            \begin{array}{ccc}
\Phi_{\rho} &  \mathbf{0} \\
\mathbf{h}_{i}^{\mathrm{tr}} & 1  \\
            \end{array}
          \right)\in\hbox{Mat}_{3}(\bar{k}[t])\cap \GL_{3}(\bar{k}(t)),\hbox{ } \Psi_{i}:=
          \left(
            \begin{array}{ccc}
\Psi_{\rho} &  \mathbf{0} \\
\mathbf{g}_{i}^{\mathrm{tr}}\Psi_{\rho} & 1  \\
            \end{array}
          \right)
\in {\rm{GL}}_{3}(\mathbb{T}).
$$
Note that by \cite[Prop. 3.1.3]{ABP04} we have $\Psi_{i}\in
\Mat_{3}(\EE)$ for each $1\leq i\leq n$. Furthermore, we have that
$\Psi_{i}^{(-1)}=\Phi_{i}\Psi_{i} $ and
\begin{equation}\label{g i theta}
 \mathbf{g}_{i}(\th)=\left(%
\begin{array}{c}
  u_{i} - \alpha_{i}\\
  -F_{\t}(u_{i}) \\
\end{array}%
\right) .
\end{equation}

 For each $1\leq i \leq n $, we note that $\Phi_i$
defines a $t$-motive $M_{i}$ with rigid analytic trivialization
$\Psi_{i}$ (cf. \cite[Prop. 4.3.1]{CP08}). Moreover, $M_{i}$ is an extension of the trivial $t$-motive $\mathbf{1}$
by $M_{\rho}$, whence $M_{i}\in \hbox{Ext}_{\mathcal{T}}^{1}(\mathbf{1},M_{\rho})$. Finally we define $M$ to be the $t$-motive
as direct sum $ M:=M_{[n]}:=\oplus_{i=1}^{n}M_{i}$,
whose defining matrix is given by block diagonal matrix $\Phi:=\oplus_{i=1}^{n} \Phi_{i}\in \Mat_{3n}(\bar{k}[t])\cap \GL_{3n}(\bar{k}(t))$ with rigid analytic trivialization
$\Psi:=\oplus_{i=1}^{n}\Psi_{i}\in \Mat_{3n}(\EE)\cap \GL_{3n}(\mathbb{T})$.

% Since $M_{\rho}$ is a sub-$t$-motive diagonally embedded in $M$, we have the following exact sequence of algebraic
% group schemes over $\FF_{q}(t)$:
% $$1\rightarrow V \rightarrow \G_{\Psi}\twoheadrightarrow \G_{\Psi_{\rho}}\rightarrow 1  .$$ Since the surjective map
% $\G_{M}\twoheadrightarrow \G_{M_{\rho}}$ is the canonical projection map, its kernel $V$ has a natural additive group % structure inside the $2n$-dimensional additive group
% $$ \left\{ \left(
%             \begin{array}{ccc}
%               1 & 0 & 0 \\
%               0 & 1 & 0 \\
%               * & * & 1 \\
%             \end{array}
%           \right)\oplus \ldots \oplus \left(
%                                         \begin{array}{ccc}
%                                           1 & 0 & 0 \\
%                                           0 & 1 & 0 \\
%                                           * & * & 1 \\
%                                         \end{array}
%                                       \right)\in \GL_{3n}
%      \right\}      .$$

Note that by (\ref{Psi th}) and (\ref{g i theta}), we have that
$$ \bar{k}(\Psi(\th))=\bar{k}( \o_1,\o_2,F_{\t}(\o_1),F_{\t}(\o_2), u_1,\ldots,u_n,F_{\t}(u_1),\ldots,F_{\t}(u_{n})   )     .$$ Thus, by Theorem \ref{T:GalThy} we have the equality
$$ \hbox{dim }\G_{\Psi}=\hbox{tr.deg}_{\bar{k}}\hbox{ } \bar{k}( \o_1,\o_2,F_{\t}(\o_1),F_{\t}(\o_2), u_1,\ldots,u_n,F_{\t}(u_1),\ldots,F_{\t}(u_{n})   )       .$$

  Note that for any $\beta\in \End(\rho)$ and $u\in \CC_{\infty}$ with
$\exp_{\rho}(u)\in \bar{k}$, by \cite[(3.13)]{BP02} we have
$$ F_{\t}(\beta u)\in \bar{k}\hbox{-Span}\left\{ 1,u,F_{\t}(u) \right\}.   $$
Thus, using (\ref{exp nu F nu}) we see that Theorem \ref{thm 2
intro} is a consequence of the following theorem, which will be proved in $\S \ref{proof of main thm sec 4}$.
\begin{theorem}\label{main thm sec 4}
 Let $\rho$ be a rank $2$ Drinfeld
$\FF_q[t]$-module
  with complex multiplication given by $\rho_{t}=\theta+ \kappa \tau
  +\tau^{2}$, $\kappa \in \bar{k}$. Suppose that the CM field of  $\rho$ is separable over $k$. Let $\o_{1},\o_{2}$ be generators of the period lattice $\L_{\rho}$ of $\rho$ over $\FF_{q}[\th]$.  Suppose that $u_{1}, \dots, u_{n}\in
  \CC_{\infty}$ satisfy $\exp_{\rho}(u_{i}) =: \alpha_{i}\in
  \bar{k}$ for $i=1, \dots, n$ and that $\omega_{1},
  u_{1}, \dots, u_{n}$ are linearly independent over
  the CM field of $\rho$. Let $\Psi$ be defined as above.   Then we have ${\rm{dim}}\hbox{ }\G_{\Psi}=2+2n$. In particular, the $2+2n$
  quantities
\[
\omega_{1}, u_{1}, \dots, u_{n}, F_{\tau}(\omega_{1}),
F_{\tau}(u_{1}), \dots, F_{\tau}(u_{n})
\]
are algebraically independent over $\bar{k}$.
\end{theorem}

As a consequence, Theorem
\ref{thm 3 intro} can be proved as follows.

\subsection{Proof of Theorem
\ref{thm 3 intro}} Without loss of generality, we let $\o=a\o_1+b
\o_2$ for some $a,b\in \FF_{q}[\th]$ with $b\neq 0$. For any $\o'
\in \L_{\rho}$ and $u\in \CC_{\infty}$, we set $\l(\o',u):=\o'
F_{\t}(u)-u F_{\t}(\o')$. Note that by (\ref{E:Legendre}) and the fact that the restriction of $F_{\t}$ to
$\L_{\rho}$ is $\FF_{q}[\th]$-linear, we have
$$
\begin{array}{rl}
  L & :=\bar{k}\left(
\begin{array}{ccccc}
  \o_1, & \o_2, & u_1, & \ldots, & u_n ,\\
  F_{\t}(\o_1), & F_{\t}(\o_2), & F_{\t}(u_1), & \ldots, & F_{\t}(u_n) \\
\end{array}   \right) \\
   &= \bar{k}\left(
\begin{array}{ccccc}
  \o_1, & \o_2, & \l(\o_1,u_1), & \ldots, & \l(\o_1,u_n), \\
  F_{\t}(\o_1), & F_{\t}(\o_2), & \l(\o_2,u_1), & \ldots, & \l(\o_2,u_n)      \\
\end{array}   \right) \\
   &= \bar{k}\left(
\begin{array}{ccccc}
  \o_1, & \o, & \l(\o_1,u_1), & \ldots, & \l(\o_1,u_n), \\
  F_{\t}(\o_1), & F_{\t}(\o), & \l(\o,u_1), & \ldots, & \l(\o,u_n)      \\
\end{array}
    \right). \\
\end{array}
$$

Let $K$ be the fraction field of $\hbox{End}(\rho)$ and let $r$ be
the $K$-dimension of the vector space over $K$ spanned by $\left\{
\o_1,\o_2,u_1,\ldots,u_n \right\}$. Then we have $r\geq n+1$ since
$\o,u_1,\ldots,u_n$ are linearly independent over
$\hbox{End}(\rho)$. Since we assume $p\neq 2$, $K$ is separable over $k$ if $\rho$ has complex multiplication. Combining Theorem \ref{main thm sec 4} and
\cite[Thm. 1.2.4]{CP08} we have
$$
\hbox{tr.deg}_{\bar{k}}\hbox{ }L=
\begin{cases}
2n+2 & \hbox{if }\rho\hbox{ has compex multiplication}, \\
2r & \hbox{if }\rho\hbox{ has no complex multiplication}.
\end{cases}
$$Note that in the case that $\rho$ has complex multiplication, $\o_1$ is a
$K$-multiple of $\o$ and $F_{\t}(\o_1)$ is a $\bar{k}$-linear
combination of $\{  1,\o,F_{\t}(\o) \}$. Therefore, we are reduced
to considering the case that $\rho$ has no complex multiplication
and $r=n+1$.

Without loss of generality, we suppose that
$\o_1,\o,u_1,\ldots,u_{n-1}$ are linearly independent over $k$. Let
$$
a_1\o_1+a_2\o+c_1u_1+\ldots+c_n u_n=0
$$
 for $a_1,a_2,c_1,\ldots,c_n\in \FF_{q}[\th]$ with $a_1 \neq 0,c_n\neq
0$, whence
\begin{equation}\label{E:lin omega u}
b_1\o_1+b_2\o_2+c_1u_1+\ldots+c_n u_n=0,
\end{equation}
where $b_{1}:=a_{1}+a_{2}a$, $b_{2}:=a_{2}b   $.

Since $F_{\t}$ is $\FF_{q}$-linear, using the difference equation
$F_{\t}(\th z)=\th F_{\t}(z)+{\hbox{exp}_{\rho}(z)}^{q}$, one has
$F_{\t}(c_{j} u_{j})=c_{j} F_{\t}(u_{j})+\beta_{j}$ for some
$\beta_{j}\in \bar{k}$, $j=1,\ldots,n$. Hence using (\ref{E:lin
omega u}) and the analogue of the Legendre relation
(\ref{E:Legendre}) we obtain
$$\l(\o_{2},c_{n} u_{n})=b_{1}\tilde{\pi}/ \xi -c_{1}\l(\o_{2},u_{1})-\cdots- c_{n-1}\l(\o_{2},u_{n-1})+\gamma_{n} \o_{2}\hbox{ for some }\gamma_{n}\in \bar{k}.  $$
Since $\l(\o_2,c_{n} u_{n})=c_{n} \l(\o_2,u_{n})+ \beta_{n} \o_{2}$,
we have
$$
\l(\o_{2},u_{n})= \frac{1}{c_{n}} \left( b_{1}\tilde{\pi}/ \xi -
c_{1}\l(\o_{2},u_{1})-
\cdots-c_{n-1}\l(\o_{2},u_{n-1})+\eta_{n}\o_{2}\right) \hbox{ for
some }\eta_{n} \in \bar{k}.
$$
On the other hand, we also have
$$
\l(\o_{1},u_{n})=\frac{1}{c_{n}} \left(-b_{2}\tilde{\pi}/ \xi -
c_{1}\l(\o_{1},u_{1})-
\cdots-c_{n-1}\l(\o_{1},u_{n-1})+\eta_{n}'\o_{1}\right) \hbox{ for
some }\eta_{n}'\in \bar{k}.
$$
It follows that
$$\l(\o,u_{n})=a\l(\o_{1},u_{n})+b\l(\o_{2},u_{n})=\frac{1}{c_{n}} \left(ba_{1}\tilde{\pi}/\xi-\sum_{i=1}^{n-1}c_{i}\l(\o,u_{i})
+a\eta_{n}'\o_{1}+b \eta_{n} \o_{2}\right).$$ Since $u_{n}\in
k\hbox{-Span}\left\{\o_1,\o_2,u_{1},\ldots,u_{n-1} \right\}$ and
$$F_{\t}(u_{n})\in \bar{k}\hbox{-Span}\left\{
1,F_{\t}(\o_1),F_{\t}(\o_2),F_{\t}(u_{1}),\ldots,F_{\t}(u_{n-1})
\right\},$$  we have that
$$
\begin{array}{rl}
  L & =\bar{k}\left(
\begin{array}{ccccc}
  \o_1, & \o_2, & u_1, & \ldots, & u_{n-1} ,\\
  F_{\t}(\o_1), & F_{\t}(\o_2), & F_{\t}(u_1), & \ldots, & F_{\t}(u_{n-1}) \\
\end{array}   \right) \\
   &= \bar{k}\left(
\begin{array}{ccccc}
  \o_1, & \o_2, & \l(\o_1,u_1), & \ldots, & \l(\o_1,u_{n-1}), \\
  F_{\t}(\o_1), & F_{\t}(\o_2), & \l(\o_2,u_1), & \ldots, & \l(\o_2,u_{n-1})      \\
\end{array}   \right) \\
   &= \bar{k}\left(
\begin{array}{ccccc}
  \o_1, & \o_2, & \l(\o_1,u_1), & \ldots, & \l(\o_1,u_{n-1}), \\
  F_{\t}(\o_1), & \tilde{\pi}, & \l(\o_2,u_1), & \ldots, & \l(\o_2,u_{n-1})      \\
\end{array}
    \right)\\
&= \bar{k}\left(
\begin{array}{ccccc}
  \o_1, & \o, & \l(\o_1,u_1), & \ldots, & \l(\o_1,u_{n-1}), \\
  F_{\t}(\o_1), & \l(\o,u_n) , & \l(\o,u_1), & \ldots, & \l(\o,u_{n-1})      \\
\end{array}
    \right),\\
\end{array}
 $$ where the third equality uses (\ref{E:Legendre}) and the fourth
 equality uses the  assumption $ba_{1}\neq 0$. This proves Theorem \ref{thm 3 intro} since $\hbox{tr.deg}_{\bar{k}}\hbox{
 }L=2(n+1)$.

\subsection{The $\mathcal{K}$-span of $\left\{M_{i} \right\}$ in $ \hbox{Ext}_{\mathcal{T}}^{1}(\mathbf{1},M_{\rho}) $}
We continue with the notations in $\S \ref{sec main thm}$. Put $\mathcal{K}:=\hbox{End}_{\mathcal{T}}(M_{\rho})$ and let $K$ be the CM field of $\rho$. We shall note that $\hbox{Ext}_{\mathcal{T}}^{1}(\mathbf{1},M_{\rho})$ is an additive group under the Bare sum. More precisely, given
$X_{i}\in \hbox{Ext}_{\mathcal{T}}^{1}(\mathbf{1},M_{\rho})$ whose defining matrix is given by
$$\left(
    \begin{array}{cc}
      \Phi_{\rho} & \mathbf{0} \\
      \mathbf{v}_{i} & 1 \\
    \end{array}
  \right)\in \Mat_{3}(\bar{k}[t])\cap \GL_{3}(\bar{k}(t)) \hbox{ for }i=1,2, $$ then with respect to a suitable choice of $\bar{k}(t)$-basis for the Bare sum
  $X_{1}+_{B}X_{2}$ the matrix representing multiplication by $\s$ on $X_{1}+_{B}X_{2}$ is given by
  $$   \left(
    \begin{array}{cc}
      \Phi_{\rho} & \mathbf{0} \\
      \mathbf{v}_{1}+\mathbf{v}_{2} & 1 \\
    \end{array}
  \right)\in \Mat_{3}(\bar{k}[t])\cap\GL_{3}(\bar{k}(t))   .$$

Further, $\hbox{Ext}_{\mathcal{T}}^{1}(\mathbf{1},M_{\rho})$ has a $\mathcal{K}$-module structure as follows. First,
let $\mathbf{m}\in \Mat_{2\times 1}(M_{\rho})$ comprise the $\bar{k}(t)$-basis of $M_{\rho}$ so that $\s \mathbf{m}=\Phi_{\rho}\mathbf{m}$. Given any nonzero element $f\in \mathcal{K}$, we have that $f (\mathbf{m})=F \mathbf{m}$ for some matrix $F\in \GL_{2}(\bar{k}(t))$. For any $X\in \hbox{Ext}_{\mathcal{T}}^{1}(\mathbf{1},M_{\rho})$ whose defining matrix is given by $$ \left(
    \begin{array}{cc}
      \Phi_{\rho} & \mathbf{0} \\
      \mathbf{v} & 1 \\
    \end{array}
  \right)\in \Mat_{3}(\bar{k}[t])\cap\GL_{3}(\bar{k}(t))   ,$$ we define $f_{*}X$ to be the pushout of the maps $f:M_{\rho}\rightarrow M_{\rho}$ and $M_{\rho}\hookrightarrow X$. With respect to a suitable choice of $\bar{k}(t)$-basis for $f_{*}X$, the matrix representing multiplication by $\s$ on $f_{*}X$ is given by $$  \left(
    \begin{array}{cc}
      \Phi_{\rho} & \mathbf{0} \\
      \mathbf{v}F & 1 \\
    \end{array}
  \right)\in \Mat_{3}(\bar{k}[t])\cap\GL_{3}(\bar{k}(t)) .$$

\begin{theorem}\label{K independence in Ext}
Let notations and assumptions be given in Theorem \ref{main thm sec 4}. For each $\a_{i}$, $1\leq i \leq n$,
let $M_{i}\in{\rm{Ext}}_{\mathcal{T}}^{1}(\mathbf{1},M_{\rho})$ be the $t$-motive associated to $\a_{i}$ defined in $\S \ref{sec main thm}$. Then $M_{1},\ldots,M_{n}$ are $\mathcal{K}$-linearly independent in ${\rm{Ext}}_{\mathcal{T}}^{1}(\mathbf{1},M_{\rho})$.

\end{theorem}
\begin{proof}
 Suppose on the contrary that $X:={f_{1}}_{*}M_{1}+_{B}\cdots+_{B}{f_{n}}_{*}M_{n} $ is trivial in $ {\rm{Ext}}_{\mathcal{T}}^{1}(\mathbf{1},M_{\rho})$ for some $f_{1},\ldots,f_{n}\in \mathcal{K}$, which are not all zero. For each $1\leq i\leq n$ with $f_{i}\neq 0$, let $F_{i}\in \GL_{3}(\bar{k}(t))$ satisfy $f_{i}(\mathbf{m})=F_{i}\mathbf{m}$ and set $(A,B):= \sum_{i=1}^{n} (\a_{i},0)F_{i}$. Then with respect to a suitable choice $\mathbf{x}\in \Mat_{3\times 1}(X)$, which comprises a $\bar{k}(t)$-basis of $X$, the matrix representing multiplication by $\s$ on $X$ is given by
$$ \Phi_{X}:= \left(
     \begin{array}{ccc}
       0 & 1 & 0 \\
       (t-\th) & -\kappa^{(-1)} & 0 \\
       A & B & 1 \\
     \end{array}
   \right)\in\Mat_{3}(\bar{k}[t]) \GL_{3}(\bar{k}(t)).
   $$Set $\mathbf{w}:= \sum_{i=1}^{n} \mathbf{g}^{\hbox{tr}}_{i} F_{i}\Psi_{\rho}.$ Then   $ \Psi_{X}:=\left(
                 \begin{array}{cc}
                   \Psi_{\rho} & 0 \\
                   \mathbf{w} & 1 \\
                 \end{array}
               \right)\in \GL_{3}(\TT)\cap \Mat_{3}(\EE)
               $ is a rigid analytic trivialization for $\Phi_{X}$.

Let $\Phi_{\rho}\oplus (1)\in \Mat_{3}(\bar{k}[t])\cap \GL_{3}(\bar{k}(t))$ be the block diagonal matrix. Since $X$ is trivial in ${\rm{Ext}}_{\mathcal{T}}^{1}(\mathbf{1},M_{\rho})$, there exists
$$\gamma=\left(
           \begin{array}{ccc}
             1 & 0 & 0 \\
             0 & 1 & 0 \\
             a & b & 1 \\
           \end{array}
         \right)\in \GL_{3}(\bar{k}(t))
  $$ so that if we take $\mathbf{x}':= \gamma \mathbf{x}$ as a new $\bar{k}(t)$-basis of $X$, then we have
  $$ \s \mathbf{x}'= \left(\Phi_{\rho}\oplus (1) \right)
   \mathbf{x}'.
   $$ That is,
   \begin{equation}\label{E: gamma Phi X}
   \gamma^{(-1)} \Phi_{X}=\left(\Phi_{\rho}\oplus (1) \right)  \gamma.
   \end{equation}   So we have the difference equation
$$ ( \gamma \Psi_{X} )^{(-1)}=(\Phi_{\rho}\oplus (1))( \gamma \Psi_{X}  ) $$
and hence by (\ref{E:Uniq of Psi}) we have
\begin{equation}\label{E: gamma delta}
 \gamma \Psi_{X}=(\Psi_{\rho}\oplus (1))\delta
\end{equation}
for some $\delta=\left(
                 \begin{array}{ccc}
                   1 & 0 & 0 \\
                   0 & 1 & 0 \\
                   c & d & 1 \\
                 \end{array}
               \right)\in \GL_{3}(\FF_{q}(t)).
$ Note that by (\ref{E: gamma delta}) we have that $a,b$ are regular at $t=\th$.   Moreover, from (\ref{E: gamma delta}) we obtain the following equation
$$ (a,b)\Psi_{\rho}+\sum_{i=1}^{n}\mathbf{g}_{i}^{\rm{tr}} \Psi_{\rho} \eta_{i}=(c,d) ,$$
where $\eta_{i}:= \Psi_{\rho}^{-1}F_{i}\Psi_{\rho} $ is regular at $t=\th$ by Proposition \ref{prop for F ij}. As $(A(\th),B(\th))=\sum_{i=1}^{n} (\a_{i},0)F_{i}(\th)$, by specializing the above equation
at $t=\th$, we have
\begin{equation}\label{E:Linear relation over K}
a(\th)-A(\th)+\sum_{i=1}^{n}( u_{i}{F_{i}}_{11}(\th)-F_{\t}(u_{i}){F_{i}}_{21}(\th) )=c(\th)\o_{1}+d(\th)\o_{2}.
\end{equation}
We claim that $a(\th)=A(\th)$. Then by Proposition \ref{prop for F ij} and the fact that $\o_{1}$ is a $K$-multiple of $\o_{2}$,  (\ref{E:Linear relation over K}) gives a non-trivial $K$-linear relation among $\left\{\o_{1},u_{1},\ldots,u_{n}\right\}$, which contradicts to the assumption.

To prove the claim, we note that (\ref{E: gamma Phi X}) gives rise to the following equation
\begin{equation}\label{E:A a b I}
(t-\th)b^{(-1)}+A=a,\hbox{ }a^{(-1)}-\kappa^{(-1)}b^{(-1)}+B=b,
\end{equation}
which imply
\begin{equation}\label{E:A a b II}
(t-\th^{(-1)}) b^{(-2)}+A^{(-1)}-\kappa^{(-1)}b^{(-1)}+B=b.
\end{equation}
Since $(A,B)=\sum_{i=1}^{n}(\a_{i},0)F_{i}$, by Proposition \ref{prop for F ij} we have that $A$ and $B$ are regular at $t=\th,\th^{q},\th^{2},\cdots$. Suppose that $b^{(-1)}$ has pole at $t=\th$, i.e., $b$ has pole at $t=\th^{q}$. By (\ref{E:A a b II}), we see that either $b^{(-1)}$ or $b^{(-2)}$ has pole at $t=\th^{q}$, i.e., either $b$ has pole at $t=\th^{q^{2}}$ or $t=\th^{q^{3}}$. By repeating the same argument, we see that $b$ has infinitely many poles among
$\left\{\th^{q},\th^{q^{2}},\th^{q^{3}},\ldots    \right\}$, which contradicts the fact $b\in \bar{k}(t)$. So $b^{(-1)}$ is regular at $t=\th$, whence we have $a(\th)=A(\th)$ by (\ref{E:A a b I}).
\end{proof}
\subsection{Proof of Theorem \ref{main thm sec 4}}\label{proof of main thm sec 4}
Let $N$ be the $t$-motive defined by
$$ \Phi_{N}:= \left(
                \begin{array}{cccc}
                  \Phi_{\rho} &  & &  \\
                   & \ddots &  &  \\
                   &  & \Phi_{\rho} &  \\
                \mathbf{h}^{\rm{tr}}_{1} & \cdots & \mathbf{h}^{\rm{tr}}_{n} & 1 \\
                \end{array}
              \right)\in \Mat_{2n+1}(\bar{k}[t])\cap \GL_{2n+1}(\bar{k}(t))
    $$with rigid analytic trivialization
    $$ \Psi_{N}:=\left(
                   \begin{array}{cccc}
                     \Psi_{\rho} &  &  &  \\
                      & \ddots &  &  \\
                      &  & \Psi_{\rho} &  \\
                     \mathbf{g}^{\rm{tr}}_{1}\Psi_{\rho} & \cdots & \mathbf{g}^{\rm{tr}}_{n}\Psi_{\rho} & 1 \\
                   \end{array}
                 \right)\in \Mat_{2n+1}(\mathbb{E})\cap \GL_{2n+1}(\mathbb{T})
       .   $$ Note that $N$ is an extension of $M^{n}_{\rho}$ by $\mathbf{1}$, which is the pullback of $M:=\oplus_{i=1}^{n}M_{i}\twoheadrightarrow \mathbf{1}^{n}$ and the diagonal embedding $ \mathbf{1}\hookrightarrow \mathbf{1}^{n}$. As the two $t$-motives $M$ and $N$ generate the same Tannakian sub-category of $\mathcal{T}$, the motivic Galois groups $\G_{M}$ and $\G_{N}$ are isomorphic and hence our task is to prove $\hbox{dim }\G_{N}=2n+2$.

       Consider the short exact sequence of algebraic group schemes over $\FF_{q}(t)$:
       \begin{equation}\label{S.E.S for Gamma N}
       \xymatrix{
         0\ar[r]& G \ar[r] &\G_{N}\ar@{->>}[r]^{\pi} &\G_{M_{\rho}} \ar[r]& 1,}
       \end{equation}
where the surjective map $\pi: \G_{N}\twoheadrightarrow \G_{M_{\rho}}$ is the canonical projection map (cf. \cite[p.22]{CP08}). Having Theorem \ref{K independence in Ext} at hand, we follow Hardouin's argument (\cite[Thm. 4.7]{Pe07}, \cite[Cor. 2.4]{H09}) to prove $\hbox{dim }G=2n$, whence we prove Theorem \ref{main thm sec 4}.

First, we note that we have the following properties:
\begin{enumerate}
\item[(I)] $M_{\rho}$ is a simple object in $\mathcal{T}$ (cf. \cite[Lem. 3.1.1]{CP08});

\item[(II)] Every $\G_{M_{\rho}}$-module is completely reducible (cf. Lemma \ref{rest of scalars} and \cite{J03}).

\item[(III)] The additive group $G$ is smooth over $\FF_{q}(t)$.

\end{enumerate}
The third property above follows by directly proving that the induced tangent map $d \pi :\hbox{Lie }\G_{N}\rightarrow \hbox{Lie }\G_{M_{\rho}}$ is surjective. The argument is similar to the proof of \cite[Prop. 4.1.2]{CP08}, so we omit the details.

Now, let $\mathbf{n}\in \Mat_{(n+1)\times 1}(N)$ comprise the $\bar{k}(t)$-basis of $N$  so that $\s \mathbf{n}=\Phi_{N}\mathbf{n}$. Note that $\Psi_{N}^{-1}\mathbf{n}$ is a canonical $\FF_{q}(t)$-basis of $N^{B}$. We set $f$ to be the last coordinate of $\Psi_{N}^{-1}\mathbf{n}$. For any $\FF_{q}(t)$-algebra $R$, using (\ref{E:GammaDef}) and Remark \ref{rmk for Galois rep} we consider the following well-defined map
$$
     \begin{array}{rrcl}
       \zeta^{(R)}: & G(R) & \rightarrow & R\otimes_{\FF_{q}(t)}(M_{\rho} ^{n} )^{B}\\
        & g & \mapsto & (g-1)f .\\
     \end{array}
    $$ When we regard $(M_{\rho}^{n})^{B}$ as an additive group scheme over $\FF_{q}(t)$, the map $\zeta$ defined above gives rise a morphism of group schemes over $\FF_{q}(t)$. Moreover, since $\G_{M_{\rho}}$ has a natural action on $G$ coming from (\ref{S.E.S for Gamma N}),  one checks directly
    that $\zeta^{(R)}$ is $\G_{M_{\rho}}(R)$-equivariant. In other words, one has:
    \begin{enumerate}
\item[(IV)]    the image $\zeta(G)$ is a $\G_{M_{\rho}}$-submodule of $(M_{\rho}^{n})^{B}$.
\end{enumerate}
Further, following Hardouin (cf. \cite[Lem. 2.3]{H09}) one uses (I)$\sim$(IV) to prove:
\begin{enumerate}
\item[(V)] $\zeta(G)\cong U^{B}$, where $U$ is a sub-$t$-motive of $M_{\rho}^{n}$ so that $N/U$ is split as
direct sum of $M_{\rho}^{n}/U$ and $\mathbf{1}$.
\end{enumerate}
Hence, to prove $\hbox{dim }G=2n$, it suffices to prove $U^{B}=(M_{\rho}^{n})^{B}$.

Suppose on the contrary that $U^{B}\subsetneq (M_{\rho}^{n})^{B}$. As $U$ is a proper sub-$t$-motive of the
completely reducible $t$-motive $M_{\rho}^{n}$, there exists a non-trivial morphism $\phi\in{\rm{Hom}}_{\mathcal{T}}(M_{\rho}^{n},M_{\rho})$ so that $U\subseteq \hbox{Ker }\phi$. Moreover, the morphism $\phi$ factors through the map $M_{\rho}^{n}/U \rightarrow M_{\rho}^{n}/\hbox{Ker }\phi $ as in the following commutative diagram:
$$
\xymatrix{
M_{\rho}^{n} \ar[dr]^{\phi} \ar[d] & \\
M_{\rho}^{n}/ U \ar [r]& M_{\rho}^{n}/ \hbox{Ker }\phi \cong M_{\rho}.
}
$$
Since $\phi\in{\rm{Hom}}_{\mathcal{T}}(M_{\rho}^{n},M_{\rho})$, we can write
$\phi(m_{1},\ldots,m_{n})= \sum_{i=1}^{n} f_{i}(m_{i})$ for some $f_{1},\ldots,f_{n}\in \mathcal{K}$, not all zero. Then the pushout $\phi_{*}N={f_{1}}_{*}M_{1}+_{B}\cdots+_{B} {f_{n}}_{*}M_{n}$ is a quotient of $N/U$. By (V), it follows that $\phi_{*}N$ is trivial in
${\rm{Ext}}^{1}_{\mathcal{T}}(\mathbf{1},M_{\rho})$. But this contradicts to the $\mathcal{K}$-linear independence of
$M_{1},\ldots,M_{n}$ in $ {\rm{Ext}}^{1}_{\mathcal{T}}(\mathbf{1},M_{\rho}) $.

\begin{remark}\label{rmk final subsec}
In this remark, we shall mention that:
\begin{enumerate}

\item[(i)] One can combine Papanikolas' theory, the ABP criterion (cf. \cite[Thm. 1.3.2]{ABP04}) and Yu's sub-$t$-module theorem
(cf. \cite[Thm. 0.1]{Yu97}, \cite[Prop. 2]{Br01}) to give an alternative proof of Theorem \ref{main thm sec 4}, but we do not discuss the details here.
\item[(ii)] When $p=2$ and the Drinfeld module $\rho'$ is given by ${\rho'}_{t}=\th+ (\sqrt{\th}+\sqrt{\th^{q}})\t+\t^{2}$, then the CM field of $\rho'$ is $\FF_{q}(\sqrt{\th})$ which is inseparable over $k$. Moreover, one finds that the algebraic group $\G_{M_{\rho'}}$ is not a torus; its unipotent radical is nontrivial (see (\ref{embedding into res of scalars})). It is not clear to the author
     whether the argument above still works in this situation since it relies on the property (II), although it is believed that
      the result of Theorem \ref{thm 2 intro} holds in this case. However, we point out that every rank $2$ Drinfeld $\FF_{q}[t]$-module with complex multiplication whose CM field is inseparable over $k$ is isomorphic to $\rho'$, so it is essentially the omitted case when
    there are extra endomorphisms.

\end{enumerate}

\end{remark}

\end{document}